\documentclass[11pt]{amsart}
\usepackage[utf8]{inputenc}
\usepackage[english]{babel}
\usepackage{amsmath}
\usepackage{amsthm}
\usepackage{amsfonts}
\usepackage{amssymb}
\usepackage{esdiff}
\usepackage{mathtools} 
\usepackage{xcolor} 
\usepackage{bbm} 
\usepackage{dsfont} 
\usepackage{diagbox}
\usepackage{subcaption} 
\usepackage{pdfsync}
\usepackage{amsthm} 

\newtheorem{thm}{Theorem}
\newtheorem*{thm*}{Theorem}
\newtheorem{thmx}{Theorem}

\newtheorem{prop}{Proposition}
\newtheorem{cor}{Corollary}
\newtheorem{lem}{Lemma}
\newtheorem{conj}{Conjecture}
\newtheorem{fact}{Fact}
\newtheorem{question}{Question}

\makeatletter
\newtheorem*{rep@theorem}{\rep@title}
\newcommand{\newreptheorem}[2]{%
\newenvironment{rep#1}[1]{%
 \def\rep@title{#2 \ref{##1}}%
 \begin{rep@theorem}}%
 {\end{rep@theorem}}}
\makeatother
\newreptheorem{theorem}{Theorem}
\newreptheorem{corollary}{Corollary}

\theoremstyle{definition}
\newtheorem{dfn}[thm]{Definition}

\newtheorem{claim}[thm]{Claim}

\newcommand{\Vol}{\text{Vol}}

\newcommand{\Isop}{\text{Isop}}
\newcommand{\N}{\mathbb{N}}
\newcommand{\Z}{\mathbb{Z}}

\newcommand{\R}{\mathbb{R}}
\newcommand{\C}{\mathbb{C}}

\newcommand{\conbod}{\mathcal{K}^n}

\newcommand{\Sph}{\mathbb{S}^{n-1}}

 \usepackage[bookmarks=true,
bookmarksnumbered=true, breaklinks=true,
pdfstartview=FitH, hyperfigures=false,
plainpages=false, naturalnames=true,
colorlinks=true,
pdfpagelabels]{hyperref}

\begin{document}

\title[Bezout for mixed volumes]{Extremizers in Soprunov and Zvavitch's Bezout inequalities for mixed volumes}
\author[M. Szusterman]{Maud Szusterman*}
\address[Maud Szusterman]{Department of Mathematics\\ Universite Paris Cit\'e\\ France}
\thanks{ * The author was supported by the National Science Foundation under Grant DMS-1929284 while the author was in residence at the Institute for Computational and Experimental Research in Mathematics in Providence, RI, during the Harmonic Analysis and Convexity program.}

\email{maud.szusterman@imj-prg.fr}

\maketitle



\begin{abstract}
    In \cite{SZ}, Soprunov and Zvavitch have translated the Bezout inequalities (from Algebraic Geometry) into inequalities of mixed volumes satisfied by the simplex. They conjecture this set of inequalities characterizes the simplex, among all convex bodies in $\R^n$. Together with Saroglou, they proved the characterization among all polytopes \cite{SSZ1} and, for a larger set of inequalities, among all convex bodies \cite{SSZ2}. The conjecture remains open for $n\geq 4$. In this work, we investigate necessary conditions on the structure of the boundary of a convex body $K$, for $K$ to satisfy all inequalities. In particular, we obtain a new solution of the $3$-dimensional case.
\end{abstract}


\section{Introduction}
\label{sec:intro}

We denote by $\mathcal{K}^n$ the class of compact convex sets in $\R^n$ : if $K\in \mathcal{K}^n$ has non-empty interior, then $K$ is called a convex body. We denote $\mathcal{K}_0^n$ the subclass of $\mathcal{K}^n$, whose elements are convex bodies containing the origin in their interior. If $K, L\in \mathcal{K}^n$, denote $K+L=\{a+b, a\in K, b\in L\}$ their Minkowski sum. We identify $\R^n$ as a subset of $\mathcal{K}^n$ (corresponding to singletons), for instance, if $x\in \R^n$, then $K+x$ denotes the corresponding translate of $K$.

 We denote by $\Vol_n(K)$ the $n$-dimensional volume of $K$, i.e. its Lebesgue measure (in $\R^n$), and we denote $B_2^n$ the unit Euclidean ball in $\R^n$. We use the notation $\kappa_n=\Vol_n(B_2^n)$.

\noindent J. Steiner has shown (see \cite[p.223]{Sch1}) that there exists a family of coefficients $(w_k)_{k=0}^n \in \R_+^{n+1}$, depending only on $K$,  such that, for any $t>0$, 
\begin{equation}
\label{steiner}
\Vol_n(K+t B_2^n)=\sum_{k=0}^n {n\choose k} w_k t^{n-k}.
\end{equation}
The latter is known as Steiner's formula. The real numbers $(w_k)$ are called quermassintegrals, and contain a lot of geometric information on $K$: 
for instance, one may derive from \eqref{steiner} that $w_n=\Vol_n(K)$, $w_0=\kappa_n$, that $w_{n-1}=\frac{1}{n} \Vol_{n-1}(\partial K)$, where $\partial K$ denotes the (topological) boundary of $K$, and $\Vol_{n-1}(\partial K)$ is the surface area of $K$; or that $w_1=\kappa_n w(K)$ is proportionnal to the mean-width $w(K)$.

Minkowski has proved that a similar formula holds when replacing $B_2^n$ by an arbitrary convex body $L$: $\Vol_n(K+tL)$ is a polynomial in $t$, of degree $n$, with non-negative coefficients. For $0\leq k\leq n$, the coefficient of $t^{n-k}$ is usually denoted ${n \choose k} V_n(K[k],L[n-k])$, where $V_n(.)$ is the mixed volume functional, defined on $(\mathcal{K}^n)^n$, and where $(K[k], L[n-k])$, means $K$ appears $k$ times as an argument, and $L$, $n-k$ times. (see Section \ref{sec_notation} page $4$ for more details).

More generally, Minkowski has proven that, for any $m\geq 2$, and any $(K_1, ... , K_m)\in \left(\mathcal{K}^n\right)^m$, $\left((t_i)_{1\leq i\leq m} \mapsto \Vol_n(t_1 K_1 +...+t_m K_m) \right)$, is a (homogeneous) polynomial whose coefficients only depend on $(K_1, ... , K_m)$. This theorem allows to define the mixed volume of an arbitrary tuple $(K_1, ... , K_n) \in (\mathcal{K}^n)^n$, as $V_n(K_1, ... , K_n)=\frac{1}{n!} \frac{\partial P}{\partial t_1 ... \partial t_n}$, where $P(t_1, ... , t_n)=\Vol_n(t_1K_1+ ... +t_n K_n)$. This definition led to the birth of the Brunn-Minkowski theory (see for instance \cite[Chapters 5-7]{Sch1}).

For any $(K,L) \in (\mathcal{K}_0^n )^2$, the sequence $(a_k)_{0\leq k\leq n}:=(V_n(K[k],L[n-k]), 0\leq k \leq n)$ is log-concave, yielding in particular $a_{n-k}^n \geq a_n^{n-k} a_0^k$ for any $k=0,1, \dots, n$.  These are particular instances of Alexandrov-Fenchel inequalities, see \cite[(7.63), p.401]{Sch1}. In particular, choosing $k=1$ and $L=B_2^n$ one recovers the classical isoperimetric inequality:
\begin{equation}
\label{isoperim}
\frac{\Vol_{n-1}(\partial K)^n}{\Vol_n(K)^{n-1}} \geq \frac{\Vol_{n-1}(\mathbb{S}^{n-1})^n}{\Vol_n(B_2^n)^{n-1}}.
\end{equation}
%
Inequality \eqref{isoperim} tells us that $\rho(K):=\frac{\Vol_{n-1}(\partial K)^n}{\Vol_n(K)^{n-1}}$ is minimized when $K$ is a Euclidean ball (indeed, these are the only minimizers). Note that $\rho(K)$ can be arbitrarily large (consider for instance a cylinder of height $h\to \infty$). However, a reverse isoperimetric inequality holds: K. Ball has introduced the following affine-invariant min-ratio: ${\text{Iso}}(K)=\min_{T\in O(n)} \rho(TK)$, where $T$ runs through linear isometries $T\in O(n)$. Ball proved that ${\text{Iso}}(K)$ is maximized when $K$ is an $n$-simplex, and that simplices are the only maximizers. 

Denote by $\text{Aff}^*(\R^n)$ the group of reversible affine transforms of $\R^n$. Within the class of convex bodies, consider the equivalence relation $K' \sim K$ if $K'=TK$ for some $T\in \text{Aff}^*(\R^n)$: let us denote by $\mathcal{AK}_0^n$ the quotient set for this relation (for instance there is only one $n$-simplex in $\mathcal{AK}_0^n$). Since ${\text{Iso}}(K)={\text{Iso}}(TK)$, for any $K\in \mathcal{K}_0^n$, and any $T\in \text{Aff}^*(\R^n)$,  ${\text{Iso}}$ can be seen as defined on $\mathcal{AK}_0^n$.

Let $\Delta$ denote an arbitrary $n$-simplex in $\R^n$. 
Note that K. Ball's reverse isoperimetric inequality $\text{Iso}(K)\leq \text{Iso}(\Delta)$, can be reformulated in terms of mixed volumes: for any $K\in \mathcal{K}_0^n$, there exists an ellipsoid $\mathcal{E}$, such that:
 $$V(K[n-1], \mathcal{E})^n V(\Delta)^{n-1} \leq V(K)^{n-1} V(\Delta[n-1], \mathcal{E})^{n}.$$
In this work, we will study another set of mixed volume inequalities involving the simplex, introduced by Soprunov and Zvavitch in \cite{SZ}, which they called \emph{Bezout inequalities for mixed volumes}. More precisely, it was shown in \cite{SZ} that for any $(A,B)\in (\mathcal{K}^n)^2$, $\Delta$ satisfies:
\begin{equation}V_n(A,B,\Delta[n-2])V_n(\Delta) \leq V_n(A, \Delta[n-1])V_n(B,\Delta[n-1])
\label{eq_Bezout_simplex}
\end{equation}
where $V_n(K)$ is a shortcut for $V_n(K[n])=\Vol_n(K)$. They gave two proofs of this inequality: a short, geometric one \cite[Theorem 2.5]{SZ}, and another one based on the Bernstein-Kushnirenko-Khovanskii Theorem (see for instance \cite[Theorem 1]{EKh}).

The following inequality is known as Fenchel inequality (see \cite[Prop 2.1]{FGM} for a proof): 
\begin{equation}
\label{eqn:relaxedbezout}
\forall K, A, B \in \mathcal{K}_0^n, \hspace{3mm}
V_n(A,B,K[n-2])V_n(K) \leq 2V_n(A, K[n-1])V_n(B,K[n-1]).
\end{equation}
This inequality is sharp: taking $K=O_n=\{x\in \R^n: \sum |x_i|\leq 1\}$ to be the cross-polytope, and $A, B$ two well-chosen segments, then equality in \eqref{eqn:relaxedbezout} holds (nota bene: for any $K\in\mathcal{K}_0^n$, equality in \eqref{eqn:relaxedbezout}, can only hold if both $A,B$ are segments, see Lemma \ref{fenchel} below).

This motivates the following definition: 
$$b_2(K):=\max \frac{V_n(A,B,K[n-2])V_n(K)}{V_n(A, K[n-1])V_n(B,K[n-1])} ,$$ where the maximum is over pairs $(A,B)\in (\mathcal{K}_0^n)^2$ (or equivalently, over pairs $(A,B)\in (\mathcal{K}^n \setminus \R^n)^2$). Taking $A=B=K$, one sees that $b_2(K)\geq 1$. On the other hand, Lemma \ref{fenchel} implies that $b_2(K)\leq 2$. 

Observe that $b_2$ is affine-invariant, and that it may also be defined as the least $b\geq 1$ such that $F_{K,b}\geq 0$, where $F_{K,b}$ is the bilinear form on $(\mathcal{K}^n)^2$ defined by $F_{K,b}(A,B)=b V_n(A, K[n-1])V_n(B,K[n-1]) -V_n(A,B,K[n-2])V_n(K)$.

With these notations, the Conjecture $1.1$ from \cite{SZ} may be rephrased as follows.

\begin{conj}
\label{conj:main}
Let $K\in \mathcal{K}_0^n$ be a convex body such that $b_2(K)=1$. Then $K$ is an $n$-simplex.
\end{conj}
In other words, $\Delta$ would be the only minimizer of $b_2$, among all convex bodies in $\R^n$. Up to the best of the author's knowledge, no convex body (seen in $\mathcal{AK}_0^n$) other than $O_n$, is known as a maximizer of $b_2$. Conjecture \ref{conj:main} is the primary focus of the present work.

Similarly as $b_2(K)$, let us introduce another affine-invariant \emph{Bezout} constant: define the $n$-linear form $G_{K,b}(A_1, ... , A_n)=bV_n(A_1, ... , A_n)V_n(K)-V_n(A_2, ... , A_n, K)V_n(A_1, K[n-1])$. We write $G_{K,b} \geq 0$ to mean $G_{K,b}(A_1, ... , A_n)\geq 0$ for all $(A_i)\in (\mathcal{K}^n)^n$. Define $b(K)=\inf \{b\geq 1: G_{K,b} \geq 0\}$. It was shown in \cite{SSZ2} that $b(\Delta)=1$, indeed the same geometric proof (as for $b_2(\Delta)=1$) works. In Section \ref{sec:zeros} below, we give an alternative proof, based on the Bernstein-Khovanskii-Kushnirenko's Theorem\cite{Be,Khov,Kush}.

In \cite{SSZ2}, Saroglou, Soprunov and Zvavitch have obtained the following characterization:
\begin{thm}
\label{thm:b(K)}
Let $K\in \mathcal{K}_0^n$ be such that $b(K)=1$. Then $K$ is an $n$-simplex.
\end{thm}
A corollary of \cite[Theorem 1.1]{SSZ2} is that Conjecture \ref{conj:main} holds true in dimension $3$. Concerning upper bounds, Xiao has proved:

\begin{thm}
\label{thm:xiao}
For any convex body $K \in \mathcal{K}_0^n$, $b(K)\leq n$.
\end{thm}
He actually proved more general inequalities (see \cite[Theorem 1.1]{Xiao}), but the above upper bound is a very particular case, which can be directly deduced from Diskant's inequality (as Xiao did for instance to derive \cite[Ineq. $(3)$]{Xiao}): see Propositions \ref{p:diskant} and \ref{Xiaothm} below, where we also provide an example showing that $n$ is sharp.


Before stating the main result of this paper (Theorem~\ref{thm:main} below), we need one or two notations. Denote by $S_K$ the surface area measure of $K$.
Denote by $K^u=\{y\in K: \langle y,u\rangle=h_K(u)\}$, where $u\in \Sph$ and $h_K$ is the support function of $K$ (see Section \ref{sec_notation} for definitions of $h_K$ and $S_K$).
%
 Finally, we denote $\Omega_k=\{u\in \mathbb{S}^{n-1}: K^u \text{ is $k$-dimensional}\}$. In particular, $supp(S_K) \cap \Omega_0$ is the set of regular directions of the boundary, and $\Omega_{n-1}$ is the set of outer normal vectors of facets (if any) of $K$. Some existing excluding conditions (for satisfying $b_2(K)=1$) deal with $\Omega_0,$ see Theorem \ref{positivecurv} below, or with $\Omega_{n-1}$ (see Theorem \ref{infinitecond} below).
 
 The main result of this paper is the following excluding condition.

\begin{thmx}
\label{thm:main}
Assume $K\in \mathcal{K}_0^n$ is such that $b_2(K)=1$. Then $S_K(\Omega_{n-2})=0$.
\end{thmx}
%
This result, together with some previous results related to Conjecture \ref{conj:main} (see theorems \ref{thm:polytope}, \ref{positivecurv}, and \ref{thm:infinite} below), yield a very short proof of the conjecture in the $3$-dimensional case.

\begin{cor}
\label{cor:intro}
In $\R^3$, the only minimizers of $b_2$, are $3$-simplices.
\end{cor}

We already note here that if similar necessary conditions held for \emph{lower-dimensional} $\Omega_{n-k}$, the same proof (of corollary \ref{cor:intro}) would work, for $n\geq 4$. More precisely, if minimizers of $b_2$ must satisfy $S_K(\Omega_{n-k})=0$, for all $k=2, \dots, k_0$,
 then Conjecture \ref{conj:main} holds true in dimension $n$, for any $2\leq n \leq k_0+1$, see the proof of Corollary \ref{cor:intro} in Section~\ref{sec:dual}, and the remarks thereafter.


\noindent {\bf Acknowledgements} The author wishes to thank Artem Zvavitch for introducing her to Conjecture~\ref{conj:main} and for motivating discussions. Thanks also to Evgueni Abakoumov and Dylan Langharst for support along the project, and for many valuable discussions and suggestions.

\section{Notations and basic definitions}
\label{sec_notation}
If $K\in \conbod$, $t>0$, then $tK=\{tx, x\in K\}$ is called a dilate of $K$. We say $K$ and $L$ are homothetic if $L$ is a translate of a dilate of $K$, i.e. if there exists $t>0$ and $x\in \R^n$, such that $L=x+tK$. With this definition, being homothetic is an equivalence relation on $\conbod$. Sometimes we rather use the extended notion (as in \cite{Sch1}) and also say that $L$ is a homothet of $K$ when $t=0$, i.e. when $L$ is a singleton (this breaks the symmetry of the relation). Which convention is used will be clear from context.

The following theorem is due to Minkowski (see \cite{Mink}).
\begin{thm}
\label{mink1903}
 Fix $m\geq 2$ and let $K_1, ... , K_m$ be convex bodies living in $\R^n$. Then $\Vol_n(t_1 K_1 + ... + t_m K_m)$, seen as a function in $(t_1, ... , t_m)\in \R_+^m$, is a (homogeneous) polynomial. Moreover, all its coefficients are non-negative, and the coefficient of the monomial $\prod t_i^{a_i}$ only depends on $\{K_i : a_i \neq 0\}$.
\end{thm}
In other words, Minkowski's theorem states that there exists a family of (non-negative) real numbers $(w_a)_a$, indexed by $a\in \{\alpha\in \N^m: \sum \alpha_j=n\}$, only depending on $(K_1, ... , K_m)$, such that, for all $t_i>0$:
$$
\Vol_n(t_1 K_1 + ... + t_m K_m)=\sum_{a_1+ ... + a_m=n} {n \choose a} w_a t^a =\sum_{a_1+ ... + a_m=n} \frac{n!}{a_1 ! ... a_m !} w_a \left( \prod_{i=1}^m t_i^{a_i} \right) .
$$
%
The coefficients $w_a$ in this polynomial are called mixed volumes, and if $a=(a_1, ... , a_m)$ (with $\sum a_i=n$), then $w_a$ is usually denoted $V_n(A_1{[a_1]}, ... , A_m{[a_m]})$. If $m=n$ and $a_i=1$ for every $i,$ one usually writes $V_n(A_1,\dots,A_n).$ Likewise, one usually writes $V_n(A_1[n-2],A_2[1],A_3[0], A_4[1])=:V_n(A_1[n-2],A_2,A_4).$ 

An immediate consequence of Theorem \ref{mink1903} is the following polarization formula:
\begin{equation}
    \label{interpol}
    n! V_n(A_1, ... , A_n)=\sum_{\emptyset \neq J \subseteq [n]} \Vol_n(A_J) \hspace{3mm} \text{where $A_J=A_{j_1}+...+A_{j_k}$, if $J=\{j_1, ... , j_k\}$.}
\end{equation}

Here are some basic properties of mixed volumes ($(a), (b)$ and $(e)$ follow from (\ref{interpol}) and from basic properties of the Lebesgue measure in $\R^n$, while $(c)$ and $(d)$ are easily deduced from Theorem \ref{mink1903}). 

\begin{enumerate}
\item[(a)] translation invariance: $V(K_1, ... , K_n)=V(x+K_1, K_2, ... , K_n)$, for any $x\in \R^n$ 
\item[(b)] symmetry in the arguments: for any $\sigma \in \mathcal{S}_n$ and for any $K_1, ... , K_n \in \conbod$, one has: $V(K_1, ... , K_n)=V(K_{\sigma(1)}, ... , K_{\sigma(k)})$ 
\item[(c)] (positive) homogeneity: for any $t_1, ... , t_n>0$, it holds that
$$V(t_1 K_1, ... , t_n K_n)= \left( \prod_i t_i\right) V(K_1, ... , K_n) $$
\item[(d)] multilinearity: for any $A,B, L_2, ... , L_n \in \conbod$, , it holds that:
$$V(A+B, L_2, ... , L_n)=V(A, L_2, ... , L_n) +V(B, L_2, ... , L_n).$$
\item[(e)]continuity (with respect to Hausdorff distance)
\end{enumerate}

Thanks to $(b)$ and $(d)$, the following polarization formula\footnote{note that case $k=0$ corresponds to (\ref{interpol})} holds, for any $0\leq k\leq n-1$, for any convex bodies $A_1, ... , A_k$, $B_1, ... ,B_{n-k}$:
\begin{equation}
    \label{interpolation}
    (n-k)! V_n(A_1, ..., A_k, B_1, ...,B_{n-k})=\sum_{\emptyset \neq J \subseteq [n-k]} V_n(A_1, ... , A_k, B_J[n-k]) \hspace{3mm} \text{where $B_J=\sum_{i\in J} B_i$}.
\end{equation}
Among other well-known properties satisfied by mixed volumes (we refer to  \cite[Chap. 4]{Sch1} for proofs and more details) are non-negativity and monotonicity:

\begin{enumerate}
\item $V(K_1, ... , K_n) \geq 0$ (non-negativity)
\item if $K'\subset K_1$, then $V(K', K_2, ... , K_n) \leq V(K_1, ... , K_n)$ (monotonicity)
\end{enumerate}

The Alexandrov-Fenchel inequality (see \cite{Al1} or \cite[p.393]{Sch1}) asserts that for any $(K_3, ... , K_n)$, the quadratic functional $V(\cdot,\cdot, K_3, ... ,K_n)$ is hyperbolic, in the sense that (for any $K_1, K_2 \in \mathcal{K}_0^n$):
$$V(K_1, K_1, K_3, ... , K_n) V(K_2, K_2, K_3, ... , K_n) \leq V(K_1, K_2, ... , K_n)^2. $$

Denote $\langle x,y\rangle=\sum_i x_i y_i$ the usual scalar product on $\R^n$. If $K$ is a convex body in $\R^n$, recall that the support function of $K$ is $h_K(x)=\max_{y\in K} \langle y,x \rangle$. Note that for any $t>0$, $x\in \R^n$, one has $h_K(tx)=th_K(x)$. Hence, $h_K$ is also often seen as a function on $\mathbb{S}^{n-1}:=\partial B_2^n$. Note that if $0 \in K$, then $h_K(u)\geq 0$ for any $u\in \mathbb{S}^{n-1}$. Likewise, if $K\in \conbod_0$, then $\min_{\mathbb{S}^{n-1}} h_K >0$.
%

 We denote
$K^{u}=\left\{x \in K:\langle x, u\rangle=h_{K}(u)\right\}$. Note that $(K+L)^u=K^u +L^u$, for any $K, L \in \mathcal{K}_0^n$, and any $u\in \Sph$. In particular, when $P$, $Q$ are two polytopes, $u$ does not a priori need to be an outer normal vector of $P$ or $Q$, to be one of the polytope $P+Q$. When $B$ is a subset of \(\mathbb{S}^{n-1}\), define the inverse spherical image \(\tau(K, B)\) of $B$ with respect to $K$: $\tau(K, B)= \bigcup_{u\in B} K^u$.

\noindent The surface area measure \(S_{K}(\cdot)\) of \(K\) is the Borel measure on \(\mathbb{S}^{n-1}\) such that $
S_{K}(B)=\mathcal{H}^{n-1}(\tau(K, B))$
for any Borel subset \(B \subset \mathbb{S}^{n-1}\). Here \(\mathcal{H}^{n-1}(\cdot)\) stands for the \((n-1)\)-dimensional Hausdorff measure. We refer to \cite[p.214, p.279]{Sch1} for more on the surface area measure of a convex body.

\begin{fact}[Minkowski's uniqueness theorem, Theorem 8.3.1 in \cite{Sch1}]
\label{fact:surfaceunique}
The measure $S_K$ is uniquely determined by $K$ (up to translation), in the sense that $S_K=S_L$ implies $L=K+x$ for some $x\in \R^n$. 
\end{fact}

Recall that if $\Omega$ is a closed subset of $\Sph$, and $g$ is a continuous function on $\Omega$, the Wulff-shape with respect to $(\Omega,g)$ is defined as $W(\Omega,g)=\bigcap_{u\in\Omega} \{x\in \R^n: \langle x,u\rangle \leq g(u)\}$; this defines a (possibly empty) compact convex set in $\R^n$; note that $W(\Omega,g)$ is a convex body if $g>0$ on $\Omega$. We refer to \cite[pp. 410-5]{Sch1} for more details on Wulff shapes. Here, if we assume $K$ is a convex body, that $\Omega$ is a closed subset of $\Sph$ such that $\text{supp}(S_K)\subset \Omega$, and that $f: \Omega \to \R$ is a continuous function: we denote by $(W_t)_t$ the family of Wulff-shape perturbations associated with $(\Omega, f)$, i.e. $W_t:= W(\Omega, h_K+tf)$. Note that $\Vol_n(W_t)>0$  when $|t|$ is small enough, i.e. if $I=I(K,f)$ denotes the open interval such that $W_t$ is a convex body, one has $0\in I$. In the sequel, we may abbreviate $(W_t)_t$ for $(W_t)_{t\in I}$.
%

For instance, if $K=P=\bigcap_{j=1}^N H^-(u_j, h_j)$ is a polytope whose set of outer normal vectors is $E(P)=\{u_j, j\leq N\}$, if $\chi_v$ denotes the characteristic function of $v\in \Sph$, then choosing $f_j=\chi_{u_j} \in \mathcal{C}(E(P), \R)$ yields the family of perturbated polytopes $(P_{j,s})_s$, with $P_{j,s}=\left(\bigcap_{k\leq N, k\neq j} H^-(u_k, h_k)\right) \cap H^-(u_j, h_j+s)$. Another example: assume $0\in int(K)$, then (as long as $supp(S_K)\subset \Omega$), $(1-Ct)K \subset W(\Omega, h_K+tf)\subset (1+Ct)K$ for any small enough $t>0$, where $C=\max |f| / \min h_K$.

The following theorem is known as Aleksandrov's variational lemma. We refer to \cite{Al1} for a proof, see also \cite[Lemma 7.4.3]{Sch1}. 

\begin{thm}
Assume $K$ is a convex body, $supp(S_K)\subset \Omega$, and $f\in \mathcal{C}(\Omega, \R)$. For $t\in\R$, denote $W_t=W(\Omega, h_K+tf)$. Then $(t\mapsto \Vol_n(W_t))$ is differentiable at $0$, and
\begin{equation}
 \label{eq_Aleksandrov_vol}
\diff{\Vol_n(W_t)}{t}\bigg|_{t=0}=\lim_{t\to 0}\frac{\Vol_n(W_t)-\Vol_n(K)}{t}=\int_{\mathbb{S}^{n-1}}f(u)dS_K(u),
  \end{equation}
\end{thm}

Minor modifications of the proof of the above theorem, yields a similar statement in terms of (first) mixed volumes, as follows (see \cite[Lemma 3]{S}).

\begin{lem}[Alexandrov's variational lemma, mixed volume version]
\label{lem:variationalmixed}
Assume that $K$ is a convex body, $supp(S_K)\subset \Omega$, and $f\in \mathcal{C}(\Omega, \R)$. Let $W_t=W(\Omega, h_K+tf)$, $t\in I$, be the associated Wulff-shape perturbations. Denote $V_1(t):=V(W_t,K[n-1])$. Then $(t\mapsto V_1(t))$ is differentiable at $0$, and:
\begin{equation} 
\label{eq_Aleksandrov}
\diff{V_1(t)}{t}\bigg|_{t=0}=\lim_{t\to 0}\frac{V_1(t)-V_n(K)}{t}=\frac{1}{n}\int_{\mathbb{S}^{n-1}}f(u)dS_K(u),
 \end{equation}
\end{lem}

Given $K, \Omega$ and $f$ as in Lemma \ref{lem:variationalmixed}, and $(W_t)_{t\in I}$ the associated family of Wulff-shape perturbations, one may easily check that for any $u\in \Sph$, the map $(t\mapsto h_{W_t}(u))$ is concave on $I$. In particular, this map is both left and right-differentiable at $t=0$. In fact, Lemma \ref{lem:variationalmixed} allows to draw a more precise conclusion here.


\begin{lem}
\label{lem:pointwiseCV}
Let $(W_t)_t$ be Wulff-shape perturbations of a given convex body $K$, with respect to $(\Omega, f)$. Then for $S_K$-almost every $u\in \Sph$:
\begin{equation} \label{pointwiseCV}
\diff{h_{W_t}(u)}{t}\bigg|_{t=0}=\lim_{t\to 0}\frac{h_{W_t}(u)-h_K(u)}{t}=f(u).
\end{equation}
\end{lem}

This is a simple consequence of Lemma \ref{lem:variationalmixed}. This lemma already appears in \cite{SSZ2}, see [Theorem 3.5] therein and its proof.

When $P$ is a polytope, denote $E(P)$ the set ot its outer normal vectors, then $S_P$ is a discrete measure on $\mathbb{S}^{n-1}$, given by
$\sum_{u\in E(P)} \Vol_{n-1}(P^u) \delta_u$, where $\delta_v$ denotes the Dirac measure at $v$. Then the first mixed volume between a convex body $L$ and the polytope $P$ can be expressed as:
\begin{equation}
\label{eq_polytope}
V(L,P[n-1])=\frac{1}{n} \sum_{u \in E(P)} h_L(u) V_{n-1}(P^u).
\end{equation}

This integral representation also holds for $K\in \conbod_0$: by setting $f=h_L$ in \eqref{eq_Aleksandrov}, one obtains that\footnote{note that $S_K$ is the only Borel measure on $\Sph$ such that (\ref{eqn:integralformula}) holds (for any convex body $L$), i.e. (\ref{eqn:integralformula}) is equivalent to the definition via inverse spherical images given on top of p.$5$. (this is easily deduced from Lemma \ref{lem:variationalmixed}).}
\begin{equation}
\label{eqn:integralformula}
V(L,K[n-1])=\frac{1}{n} \int_{\mathbb{S}^{n-1}} h_L(u) dS_K(u).
\end{equation}

More generally, under mild assumptions, if $K_2, ... , K_n \in \conbod$ 
, then there exists a unique measure $S(K_2, ... , K_n, .)$ on the sphere, such that for all convex bodies $L$, the mixed volume can be represented:
$$
V(L,K_2, ... , K_n)=\frac{1}{n} \int_{\mathbb{S}^{n-1}} h_L(u) dS(K_2, ... , K_n,u).
$$

When $P_2, ... , P_n$ are polytopes, denote $E(P_2, ... , P_n)=\{u\in \mathbb{S}^{n-1} :V_{n-1}(P_2^u, ... , P_n^u)>0 \}$. Then the mixed surface area measure $S(P_2, ... , P_n,.)$ is a discrete measure on $\mathbb{S}^{n-1}$, given by:
$$
V(L,P_2, ... , P_n)=\frac{1}{n} \sum_{u\in E} h_L(u) V_{n-1}(P_2^u, ... , P_n^u) \text{ where $E=E(P_2, ... , P_n)$}.
$$
(this can be deduced from \eqref{eq_polytope} by polarization \footnote{to be more precise: use  (\ref{interpolation}) in dimension $n$ with $k=1$, use $(P+Q)^u=P^u+Q^u$, and then (\ref{interpol}) in dimension $(n-1)$.}).


Let $\Omega=supp(S_K)$ denote the support of the surface area measure of $K$. Several times in the next sections, we partition $\Omega$ according to the dimension of the faces $K^u$, i.e. write $\Omega=\Omega_0 \cup ... \cup \Omega_{n-1}$, where $\Omega_k=\{u\in \Omega: K^u \text{ is $k$-dimensional }\}$.
For instance, if $K=P$ is a (full-dimensional) polytope in $\R^n$, then $\Omega=\Omega_{n-1}$ is the set of outer normal vectors of $K$, while if $K=B_2^n$, then $\Omega=\Omega_0=\mathbb{S}^{n-1}$.

\vspace{2mm}

Mixed volume behaves quite nicely with segments, as expressed in the next formulas.
\begin{enumerate}
     \item[a-] For $u\in \mathbb{S}^{n-1}$, and for any $K\in\conbod$, one has:
$V_n(K+t[0,u])=V_n(K)+t V_{n-1}(\pi_{u^{\perp}}(K))$,  where $\pi_{u^{\perp}}$ is the orthogonal projection onto the linear hyperplane $u^{\perp}$.
It readily follows that:
\begin{equation}
\label{eqn:proj1}
V(K[n-1],[0,u])=\frac{1}{n} \Vol_{n-1}\left( \pi_{u^{\perp}}(K)\right) .
\end{equation}
\noindent By polarization, it also holds that for any $K_1, ... , K_{n-1} \in \conbod$, one has:
$$V(K_1, ... , K_{n-1},[0,u])=\frac{1}{n} V_{n-1}\left( \pi_{u^{\perp}}(K_1), ... , \pi_{u^{\perp}}(K_{n-1})\right) .
$$
\item[b-] Fix two directions $u$ and $v$ on $\mathbb{S}^{n-1}$, and set $L_1=[0,u]$, $L_2=[0,v]$. Then, using the above formula twice yields:
\begin{equation}
\label{eqn:proj2}
    V(L_1, L_2, K{[n-2]})=\frac{1}{n(n-1)} V_{n-2}(\pi_{U^{\perp}}K)|det_2(u,v)|=\frac{2}{n(n-1)} V_{n-2}(\pi_{U^{\perp}}K) V_2(L_1, L_2) \, \,
\end{equation}
 where $U^{\perp}$ denotes the $(n-2)$-dimensional space orthogonal to $u$ and $v$ (if $u$, $v$, are colinear, then the three terms vanish).
 \item[c-]  Fix $k$ linearly independent directions $u_j\in \Sph$, $j\in \{1,...,k\}$, then a similar formula holds:
 $$V_n([0,u_1], ... , [0,u_k], K_{k+1}, ... , K_n)=\frac{k! V_k([0,u_1], ... , [0,u_k])}{n(n-1) ... (n-k+1)} V_{n-k}\left( \pi_{U^{\perp}} K_{k+1}, ... , \pi_{U^{\perp}} K_n\right),$$
 where $U^{\perp}$ denotes the $(n-k)$-dimensional space orthogonal to $u_1, ... , u_k$ (if $u_1, ... , u_k$ are not linearly independent, then both sides are zero).
\end{enumerate}

Another well-known property of mixed volumes is the following: let $K_1,... , K_n \in \conbod$, and let $T\in \text{Aff}^*(\R^n)$. Then $V(TK_1, ... , TK_n)=|det(T)| (K_1, ... , K_n).$ It follows that $b_2(K)$ and $b(K)$ (see Section \ref{sec:intro} for the definition of these two) are affine-invariant.

Finally, we will need some notations on the sphere: if $u,v\in \mathbb{S}^{n-1}$, then $dist(u,v)$ is the distance between $u$ and $v$ on the sphere: for instance dist$(u, -u)=\pi$. If $A\subset \mathbb{S}^{n-1}$ and $\epsilon>0$, we write $A^{\epsilon}:=\{u\in \mathbb{S}^{n-1}: dist(u,A)\leq \epsilon\}$. If $u\in \Sph$, and $\epsilon>0$, then $U(u,\epsilon)$ is called a cap (on $\Sph$), centered at $u$ and of radius $\epsilon$, and is defined as $U(u,\epsilon)=\{w\in \Sph: dist(u,w)\leq \epsilon\}$. Let $\epsilon, t>0$, and assume $t\epsilon <\pi$: then if $u_0\in \Sph$ and $U=U(u_0, \epsilon)$, we denote $tU:=U(u_0,t\epsilon)$.

\section{Zeroes of polynomials and mixed volumes}
\label{sec:zeros}
In this section, we outline the proofs of Equation~\eqref{eq_Bezout_simplex}. First, we must recall facts from algebraic geometry. Let $P,Q \in \C[X,Y]$ be two polynomials, of respective degree $d_1$, $d_2$. We denote $Z_P \subset (\C\setminus \{0\})^2$ the set of zeroes of $P$: $Z_P=\{(x,y)\in (\C\setminus \{0\})^2: P((x,y))=0\}$.When $P$ and $Q$ are coprime, then $Z_P \cap Z_Q$ is finite and Bezout's theorem \cite[Proposition 8.4]{F} gives an upper bound on the cardinality of the intersection, by the product of the degrees:
$$|Z_P \cap Z_Q| \leq d_1 d_2.
$$
\noindent More generally, if $Q_1, ... , Q_n \in \C[x_1, ... , x_n]$ are $n$ polynomials defined in $\C^n$ with no non-trivial common factor, then $Z_{Q_1} \cap ... \cap Z_{Q_n}$ is a finite set in $(\C \setminus \{0\})^n$, and Bezout's inequality upper bounds the cardinality of the intersection, by the product of the degrees of the polynomials. 
$$(B_n) \hspace{30mm} deg(Z_{Q_1} \cap ... \cap Z_{Q_n}):= |Z_{Q_1} \cap ... \cap Z_{Q_n}| \leq d_1 d_2 ... d_n .$$
Apart from its set of zeroes, another geometric object associated with a polynomial $Q\in \C[x_1, ... , x_n]$, is its Newton polytope. This polytope is defined as the convex hull of integer points $(k_1, ... , k_n)$, such that the monomial $x_1^{k_1} ... x_n^{k_n}$, has a non-zero coefficient in $P$. The Bernstein-Khovanskii-Kushnirenko Theorem (see \cite{Be,Kush,Kho}, or \cite{EKh}) allows one to express for most $n$-tuples of polynomials $(Q_1, ... ,Q_n)$, the intersection number $|Z_{Q_1} \cap ... \cap Z_{Q_n}|$ (which is also called the \emph{degree} of the variety $Z_{Q_1} \cap ... \cap Z_{Q_n}$, and coincides with the cardinality of the intersection, when the latter is finite), in terms of the mixed volume of the associated Newton polytopes:
$$deg(Z_{Q_1} \cap ... \cap Z_{Q_n}) = n! V(P_1, ... , P_n) \hspace{3mm} \text{where $P_i$ is the Newton polytope of $Q_i$}.$$
Though this equality doesn't hold for every $n$-tuple of polynomials $(Q_1, ... , Q_n)$, the equality holds almost surely, if $(Q_1, ... , Q_n)$ are random $n$-variate polynomials, for certain natural distributions on complex polynomials. For instance, we may first fix an arbitrary $(Q_1, ... , Q_n)$, then let $(Q'_1, ... , Q'_n)$ be an $n$-tuple of polynomials such that $Q'_i$ has same monomials $X^{\alpha}$ as $Q_i$, but with $Q'_i=\sum c'_{\alpha} X^{\alpha}$, where $c'_{\alpha}$ is a random complex number, centered at $c_{\alpha}$ (the coefficient of $X^{\alpha}$ in $Q_i$), and distributed in a small disk around $c_{\alpha}$ which does not contain $0$. While $n! V(P_1, ... , P_n)$ is fixed (and only depends on $(Q_1, ... , Q_n)$, or rather on its set of monomials), the intersection $Z_{Q'_1} \cap ... \cap Z_{Q'_n}$ is random, but almost surely the $Q'_i$ are coprime so that this is a finite intersection. The set of common zeroes $Z_{Q'_1} \cap ... \cap Z_{Q'_n} \subset (\C \setminus 0)^n$ is therefore almost surely a (random) finite set. The size of this finite set is always upper bounded by $n! V(P_1, ... ,P_n)$; the Bernstein-Khovanskii-Kushnirenko's Theorem says that this upper bound is almost surely the actual size of the set. We refer to \cite{Kho} for more details around this theorem.

\noindent When $Z_{Q_1} \cap ... \cap Z_{Q_k}$ is infinite, its  degree $deg(Z_{Q_1} \cap ... \cap Z_{Q_k})$ is nonetheless a well-defined integer, and a more general form of Bezout theorem is the following inequality:
$$deg(Z_{Q_1} \cap ... \cap Z_{Q_k}) \leq d_1 ... d_k \hspace{2mm} .$$
In this case, the theorem by Bernstein-Khovanskii-Kushnirenko still holds\footnote{(if the $Q_i$ have \emph{generic} coefficients)} and allows one to express this degree via the mixed volume of the associated polytopes: $$deg(Z_{Q_1} \cap ... \cap Z_{Q_k}) = n! V(P_1, ... , P_k, \Delta_n[n-k]), \hspace{3mm} \text{ where } \Delta_n=Conv(0,e_1, ... , e_n) .$$

As a particular case: when $k=1$, one finds $deg(Z_{Q_1})=n! V(P_1, \Delta_n[n-1])$. Moreover (almost surely) $deg(Z_{Q_1})=d_1$, the degree of $Q_1$. We now recall how this proves Equation \eqref{eq_Bezout_simplex}, see also \cite[p.2]{SZ}, before seeing how similar arguments yield $b(\Delta)=1$.

Indeed, set $k=2$ and let $P_1$ and $P_2$ be two Newton polytopes. Let $Q_1$ and $Q_2$ be two random polynomials defined over $\C^n$, defined so that $P_i$ is the Newton polytope of $Q_i$ (for instance with independent, small noise around each non-zero coefficient of some arbitrary $Q_i^{(0)}$ whose set of monomials is the set of vertices of $P_i$). Then, almost surely, the Bezout inequality and the equality in the Bernstein-Khovanskii-Kushnirenko theorem hold, yielding the following inequality in terms of mixed volumes:
$$deg(Z_1\cap Z_2) \leq d_1 d_2 \hspace{2mm} \Rightarrow \hspace{2mm} V_n(P_1, P_2, \Delta_n[n-2]) V_n(\Delta_n) \leq V_n(P_1, \Delta_n[n-1]) V_n(P_2, \Delta_n[n-1]).$$
(where we used that $\Delta_n=Conv(0,e_1, ... , e_n)$ has volume $(n!)^{-1}$).

Hence this inequality holds for any pair $(P_1, P_2)$ of polytopes with vertices in $\N^n$. By homogeneity of mixed volume, it also holds for $(\delta P_1, \delta P_2)$, meaning that it holds for pairs of polytopes with vertices in $(\delta \N)^n$. By translation invariance of mixed volume, the inequality still holds for pair of polytopes with vertices in $(\delta \Z)^n$. By continuity of mixed volume, since any pair of convex bodies $(L_1, L_2)$ can be approximated by a pair $(P_{1,\delta}, P_{2,\delta})$ of polytopes with vertices in $(\delta \Z)^n$ (with $d_{\mathcal{H}}(P_i,L_i)\leq \delta$), one concludes the inequality not only holds for pairs of Newton polytopes $(P_1,P_2)$, but actually for any pair of convex bodies $(L_1, L_2)$.


Similarly, one can deduce Bezout's inequality for mixed volumes of $n$ convex bodies, from the Bezout inequality $(B_n)$. Fix $Q_1, ... , Q_n$ some polynomials\footnote{or rather fix $P_1, ... , P_n$ first, set $A_i$ to be the set of vertices of $P_i$, then let $Q_i=\sum_{\alpha \in A_i} (2+u_{\alpha}) X^{\alpha}$ be a random  polynomial, where $u_{\alpha} \in \mathbb{D}$ are i.i.d uniform points from the unit disk, and do so independently for each $i=1,2, \dots ,n$} in $\C^n$. Let $X=Z_{Q_1}$ and $Y=Z_{Q_2} \cap ... \cap Z_{Q_n}$. Then the Bernstein-Khovanskii-Kushnirenko theorem says that, for generic choices of $Q_i$, one has:
$$
deg(X)=n! V(P_1, \Delta_n[n-1]), \hspace{3mm} deg(Y)=n! V(P_2, ... , P_n, \Delta_n), \hspace{3mm} deg(X\cap Y)=n! V(P_1, ... , P_n).
$$
where $P_i$ is the Newton polytope of polynomial $Q_i.$ Hence, the Bezout inequality $deg(X \cap Y) \leq deg(X) deg(Y)$ translates to
$$
V(P_1, ... , P_n) V(\Delta_n) \leq V(P_1, \Delta_n[n-1]) V(P_2, ... , P_n , \Delta_n).
$$
As before, using homogeneity and translation invariance of $V(.)$, and (simultaneous) approximation of a tuple of convex bodies by polytopes with vertices on a grid, the same inequality remains valid if replacing Newton polytopes $P_i$, with arbitrary convex bodies $K_i.$  In other words:%
\begin{equation}\label{eq_Bezout_n_simplex_temp}
\text{$\forall K_1, \dots , K_n \in \conbod,$ } \hspace{4mm} V(K_1, ... , K_n) V(\Delta_n) \leq V(K_1, \Delta_n[n-1]) V(K_2, ... , K_n , \Delta_n).\end{equation}
The fact that $V(TL_1, ... , TL_n)=|det(T)|V(L_1, ... , L_n)$ for any affine transform $T$ (and any $L_i \in \conbod$) implies that for any simplex $\Delta$, and for any $n$-tuple of convex bodies $(K_i)_{i\leq n}$ we have:
\begin{equation}\label{eq_Bezout_n_simplex}V(K_1, ... , K_n) V(\Delta) \leq V(K_1, \Delta[n-1]) V(K_2, ... , K_n , \Delta).\end{equation}

In order to study \eqref{eq_Bezout_n_simplex} for bodies besides the simplex, we have introduced the notation $b(K)$ (see definition $p.3$), which can be equivalently defined as the following supremum:
%
\begin{equation}b(K)=\sup_{K_1, ... , K_n} \frac{V(K_1, ... , K_n) V(K)}{V(K_1, K[n-1]) V(K_2, ... , K_n , K)},
\label{eq_bezout_constant}
\end{equation} where the supremum is over $n$-tuples of bodies with non-empty interior. In other words, $b(K)$ is the least constant $b>0$ such that, for any $K_1, \dots , K_n \in \conbod$, it holds that:
\begin{equation}\label{eq_Bezout_n}V(K_1, ... , K_n) V(K) \leq b(K) V(K_1, K[n-1]) V(K_2, ... , K_n , K).\end{equation}

Notice that $b(K)$ is affine invariant. Equation~\eqref{eq_Bezout_n_simplex} says that $b(\Delta)\leq 1$ for any $n$-simplex $\Delta$. Note that $b(K)\geq b_2(K) \geq 1$ for any $K\in \conbod_0$,  by choosing $K_3=...=K_n=K$. Therefore, we have that $b(\Delta)=b_2(\Delta)=1$ for any simplex $\Delta.$ A shorter, purely geometric proof of this fact is possible (see \cite{SSZ2}, page 9).

Let us now discuss upper bounds on $\max_K b(K)$. Monotonicity of mixed volume, the fact that $b(\Delta)\leq 1$ (for all $n$-simplices) and the fact (due to Lassak) that any convex body $K$ contains an $n$-simplex $\Delta$ with $\Delta \subset K \subset (n+2) \Delta$, immediately yields $\max_K b(K) \leq (n+2)^n$:
\begin{align*} V(L_1, ... , L_n) V(K)\leq (n+2)^n V(L_1, ... , L_n) V(\Delta) &\leq (n+2)^n V(\Delta, L_2, ... , L_n) V(L_1,\Delta[n-1])
\\
&\leq (n+2)^n V(K, L_2, ... , L_n) V(L_1,K[n-1]).
\end{align*}
(see also \cite[Prop $5.1$]{SZ}, where Lassak's result is used to upper bound some related quantities).

A somewhat better upper bound can be obtained thanks to John's theorem. Recall the inradius $r(K,L)$ of $K$ relative to $L$, is defined as:
$$r(K,L)=\max \{r>0: \exists x\in \R^n, x+rL \subset K \}.$$
Note that in the definition of $b(K)$ as a maximal ratio, the maximum can be taken over $n$-tuples $(L_1, ... , L_n)$ such that $L_1 \subset K$ and $L_1$ contains a segment of length at least $c=r(K, B_2^n)$. (this easily follows from translation invariance and homogeneity of mixed volume).


Define $c_0:=\frac{c}{n \Vol_n(K)} \min_u \Vol_{n-1}\left( \pi_{u^{\perp}}(K)\right)$, so that $c_0>0$ (if $K\in \mathcal{K}_0^n$). Note that $c_0$ only depends on $K$. Then, since $L_1$ contains some segment of length $c$, monotonicity of $V(.)$ implies:
$$V(K[n-1],L_1) \geq c \min_{u\in \mathbb{S}^{n-1}} V(K[n-1],[0,u])=\frac{c}{n} \min_u \Vol_{n-1}\left(\pi_{u^{\perp}}(K) \right) =c_0  \Vol_n(K).$$
And since $L_1\subset K,$ this latter lower bound yields:
$$V(L_1, ... , L_n)V(K)\leq V(K, L_2, ... , L_n)V(K) \leq c_0^{-1} V(K, L_2, ... , L_n) V(K[n-1],L_1), \hspace{2mm} \text{ and thus $b(K) \leq c_0^{-1}$.}$$
Since $b(K)$ is affine-invariant, we can assume $K$ is in its John's position, i.e. $B_2^n \subseteq K \subseteq \sqrt{n} B_2^n$.
In this case, $\min_u V\left(\pi_{u^{\perp}}(K)\right) \geq {\kappa_{n-1}}$ while $V(K)\leq n^{n/2} \kappa_n$, so that the above argument yields $$b(K) \leq c_0^{-1}= \frac{n \Vol_n(K)}{r(K,B_2^n) \min_u V(\pi_{u^{\perp}}(K))}\leq \frac{\kappa_n}{\kappa_{n-1}} n^{\frac{n}{2}+1} < \sqrt{2\pi} n^{\frac{n+1}{2}}$$ (where the last inequality is due to log-convexity of the $\Gamma$ function).

The best upper bound is in fact $\max_K b(K) \leq n$. It is not hard to see that $n$ is sharp (see Proposition \ref{Xiaothm} below), and the upper bound can be derived via Diskant inequality, which implies a lower bound on the inradius, as first noticed by Xiao  (see  \cite[Corollay 1.2, Inequalities $(2)$, $(3)$]{Xiao}).


\vspace{3mm}
\begin{prop}[Diskant's inequality]
\label{p:diskant}
Let $K, L$ be two convex bodies. Then:
$$V(K[n-1],L)^{n/n-1} - V(K) V(L)^{1/n-1} \geq \left[ V(K[n-1],L)^{1/n-1} - r(K,L) V(L)^{1/n-1} \right]^n.$$
\end{prop}

\begin{prop}
\label{Xiaothm}
For all $K\in \conbod_0$, $b(K) \leq n$. Moreover, this upper bound is sharp.
\end{prop}

\begin{proof}
Denote $r=r(K,L)$, $v_1=V_1(K,L)=V(K[n-1],L)$, $v_0=\Vol_n(K)$, and $v_L=\Vol_n(L)$. Finally, denote $u= \frac{v_0 v_L^{1/n-1}}{v_1^{n/n-1}}\geq 0$. Note that $u\leq 1,$ by Brunn-Minkowski's first inequality. Hence $1-(1-u)^{1/n} \geq \frac{u}{n}$.
\noindent Combining this 
with Diskant's inequality (as stated in Proposition~\ref{p:diskant}) gives:
$$r\geq \left( \frac{v_1}{v_L} \right)^{\frac{1}{n-1}} \left( 1- (1-u)^{1/n} \right) \geq \frac{1}{n} u  \left( \frac{v_1}{v_L} \right)^{\frac{1}{n-1}} = \frac{1}{n} \frac{v_0}{v_1}=\frac{1}{n} \frac{V(K)}{V(K[n-1], L)}.
$$
By multilinearity of mixed volume, we may as well replace $L_1$ with $x+r(K,L_1) L_1$ (in the definition of $b(K)$ as a maximum), and hence assume that $L\subset K$ and that $r=1$. In this case, the above inequality reads $V(K) \leq n V(K[n-1],L)$, while monotonicity of $V(.)$ gives that $V(L, L_2, ... , L_n )\leq V(K, L_2, ... , L_n )$. Since this holds true for arbitrary $L_2, ... , L_n$, it follows that $b(K)\leq n$.

\vspace{2mm}
To see that $\max_K b(K) \leq n$ is a sharp upper bound, take $K=[0,1]^n=[0,e_1]+ ... +[0,e_n]$, where $(e_i)$ denotes the canonical basis in $\R^n$. Take $L_1=[0,e_1]$, $L_2=[0,e_2]$, etc. Then
$$V(K,L_2, ... , L_n)=V(L_1+L_2+...+L_n , L_2, L_3, ... , L_n)=\sum_{j=1}^n V(L_j, L_2, ... , L_n)= V(L_1, L_2, ... , L_n)$$
since $V(L_j, L_2, ... , L_n)=0$ for any $j\geq 2$.
Meanwhile, $nV(K[n-1],L_1)=\Vol_{n-1}\left( \pi_{e_1^{\perp}}(K) \right)=1=\Vol_n(K)$. Therefore,
$$V(L_1, L_2, ... , L_n) V(K)=nV(K[n-1],L_1)V(K,L_2, ... , L_n), \hspace{5mm} \text{showing that $b\left( [0,1]^n \right) =n$}.$$
\end{proof}

\noindent Let $C=[0,1]^n$ denote the unit cube in $\R^n$. The above proof naturally brings the following question:

\begin{question}
Let $K \in \mathcal{K}_0^n$ be such that $b(K)=n$. Do we have $K=TC$ for some affine transform $C$ ?
\end{question}
\noindent In other words, is the cube the unique convex body (seen in $\mathcal{A}\mathcal{K}_0$) maximizing $b(K)$ ?

Let us briefly discuss lower bounds on $b_2(K).$ Note that the argument at the end of the above proof also yields $b_2([0,1]^n)\geq \frac{n}{n-1}$. To see it, let $L_i=[0,e_i]$, for $i=1,2, ... , n$, so that $C=[0,1]^n=L_1+...+L_n$. Then, for $i=1,2$: $V(C[n-1],L_i)=\frac{1}{n}$, while
$V(C[n-2], L_1, L_2)=\frac{1}{n(n-1)}$ and $V(C)=1$. Therefore $b_2(C)\geq \frac{n}{n-1}$. Indeed, $b_2([0,1]^n)=\frac{n}{n-1}$.

\begin{claim}
\label{claim:cube}
Let $C=[0,1]^n$ be the unit cube in $\R^n$. Then $b_2(C)=\frac{n}{n-1}$.
\label{cube_constant}
\end{claim}
\begin{proof}
We have just argued that $b_2(C)\geq \frac{n}{n-1}$, by choosing for $(L_1, L_2)$ two (non-parallel) sides of the cube. It remains to check that for any $A,B \in \conbod$: 
$$V(A,B,C[n-2])V(C) \leq \frac{n}{n-1} V(A, C[n-1])V(B,C[n-1]).
$$
Using properties $(b)$ and $(d)$ of mixed volume and that $C=[0,e_1]+...+[0,e_n]$ one gets:
$$
V(A,C[n-1])=(n-1)! \sum_{i=1}^n V(A,\left( [0,e_k]\right)_{\substack{1\leq k \leq n \\ k\neq i}} )=\frac{1}{n} \sum_i |\pi_i(A)|
$$ where $\pi_i$ is the orthogonal projection
onto $\R e_i$, and where $|A'|$ is the length of (segment) $A'$. Similarly,
$$
V(A,B,C[n-2])=(n-2)! \sum_{i<j} V(A,B,,\left( [0,e_k]\right)_{\substack{1\leq k \leq n \\ k\neq i, j}})=\frac{2}{n(n-1)} \sum_{i<j} V_2( \pi_{i,j}(A), \pi_{i,j}(B))
$$
where $\pi_{i,j}$ denotes the orthogonal projection onto $\R e_i +\R e_j$.

For all $i<j$, note that $V_2( \pi_{i,j}(A), \pi_{i,j}(B)) \leq  \frac{1}{2} \left( |\pi_i(A)| |\pi_j(B)|+|\pi_j(A)| |\pi_i(B)| \right)$, according to Lemma~\ref{lem:rectangle} (see Appendix). Therefore:

\begin{align*}V(C)V(A,B,C[n-2])&=\frac{2}{n(n-1)} \sum_{i<j} V_2( \pi_{i,j}(A), \pi_{i,j}(B)) 
\\
&\leq \frac{1}{n(n-1)} \sum_{i<j} \left( |\pi_i(A)| \pi_j(B)|+|\pi_j(A)| |\pi_i(B)|\right)
\\
&=\frac{1}{n(n-1)} \sum_{i \neq j} |\pi_i(A)| |\pi_j(B)|\leq \frac{1}{n(n-1)} \left( \sum_i |\pi_i(A)| \right)  \left( \sum_j |\pi_j(B)| \right)
\\
&=\frac{n}{n-1} V(A,C[n-1])V(B,C[n-1] ).
\end{align*}
\end{proof}

{\bf{Remark:}} It follows from this proof that equality can only hold if $A$, $B$ are totally orthogonal, or, more precisely in our case: if there exists $I=\{i_1, ... , i_k\}\subset [n]$, and $x_0, x_1 \in \R^n$, such that $A\subset E+x_0$ and $B\subset E^{\perp}+x_1$, where $E=\R e_{i_1} +...+\R e_{i_k}$.

However, $\frac{n}{n-1} \neq \max_K b_2(K)$, that is to say: the cube is not a maximizer for $b_2(.)$, since actually $\max_K b_2(K)=2$, as follows from Lemma~\ref{fenchel} (known as Fenchel's inequality) and the example afterwards. In fact, whether or not $b_2(C)$ is maximal among zonoids is related to a restricted version of a conjecture by Dembo, Cover, Thomas (see remark \ref{appendix:cube} in appendix, and see \cite{FMMZ} and references therein).

In \cite[Proposition 2.1]{FGM}, it was shown that, for any $A,B,C, K \in \mathcal{K}_0^n$, one has:
\begin{equation}
\label{eqn:FGM}
 \frac{V(B+C,B+C,K[n-2])}{V(B+C,A,K[n-2])}\geq \frac{V(B,B,K[n-2])}{V(B,A,K[n-2])}+\frac{V(C,C,K[n-2])}{V(C,A,K[n-2])}.
\end{equation}
The special case $K=A$ follows from Brunn-Minkowski's inequalities. The case with $K$ arbitrary can be derived using two Alexandrov-Fenchel inequalities, we refer to \cite[Lemma and Proposition 2.1]{FGM} for a proof (and for an even more general statement).

A consequence of \eqref{eqn:FGM} is the following, see also \cite{Fen}.

\begin{lem}
\label{fenchel}
Let $K\in \mathcal{K}_0^n$ be a convex body. If $M, L \in \mathcal{K}^n$ are compact convex sets, denote $V(M,L)=V_n(M,L,K[n-2])$.
Then for any $(M,L)\in \left(\mathcal{K}^n\right)^2$:
$$V(K,K)V(M,L)\leq 2 V(K,L)V(K,M).$$
\end{lem}
%
Actually, in case $V(M,L)=V_n(M,L,K[n-2])>0$, one actually gets the sharper inequality:
\begin{equation}
\label{eqn:fenchelsharp}
V(K,K)V(M,L)\leq 2V(K,L)V(K,M)- \frac{V(L,L)V(K,M)^2}{V(M,L)}.
\end{equation}

(to derive \eqref{eqn:fenchelsharp}, take $K=B$ in \eqref{eqn:FGM}, and $A=M$, $C=L$).

Note that Lemma \ref{fenchel} exactly tells us that for any convex body $K\in \conbod$, $b_2(K)\leq 2$.

We recall that the inequality is sharp: take $K=O_n$ the $l_1$-ball in $\R^n$. Denote $(e_i)$ the canonical basis in $\R^n$. Let $L_1$, and $L_2$ be the two segments: $L_1=[0,e_1+e_2]$, $L_2=[0,e_1-e_2]$. Then
$$V_n(L_2,K[n-1])=V_n(L_1,K[n-1])=\frac{\sqrt{2}}{n} \Vol_{n-1}(\pi K)$$
where $\pi$ denotes the orthogonal projection onto $(e_1+e_2)^{\perp}$, so that $\Vol_{n-1}(\pi K)=\frac{1}{\sqrt{2}} \Vol_{n-1}(O_{n-1})$.
On the other side (according to \eqref{eqn:proj2}), 
$$V_n(L_1,L_2,K[n-2])=\frac{2V_2(L_1,L_2)}{n(n-1)} \Vol_{n-2}(O_{n-2})=\frac{2}{n(n-1)}\Vol_{n-2}(O_{n-2})$$
since the projection of $O_n$ onto $(e_1, e_2)^{\perp}$, is $O_{n-2}$, and since $L_1$ and $L_2$ both have length $\sqrt{2}$, and are orthogonal. It follows that
$$V_n(L_1,K[n-1])V_n(L_2,K[n-1])=\frac{\Vol_{n-1}(O_{n-1})^2}{n^2}=\frac{2^{2n-2}}{(n!)^2}=\frac{1}{2}\frac{2^{2n-1}}{(n!)^2}=\frac{1}{2} V_n(L_1,L_2,K[n-2])V_n(K).$$
And hence $b_2(O_n)=2$. Assume $K\in \conbod_0$: then because of \eqref{eqn:fenchelsharp}, we may have equality in Lemma \ref{fenchel} only if $M,L$ are both segments. In other words, $b_2(K)=2$ if and only if there exists $u,v\in \Sph$, with 
$$\Vol_n(K) \Vol_{n-2}(\pi_{(u,v)^{\perp}} K) V_2([0,u],[0,v])=\frac{(n-1)}{n} \Vol_{n-1}(\pi_{u^{\perp}} K) \Vol_{n-1}(\pi_{v^{\perp}} K) 
$$
where we remind $\pi_{(u,v)^{\perp}}$ denotes the (orthogonal) projection onto $(u,v)^{\perp}$.

Although Conjecture \ref{conj:main} (characterization of $\Delta$ as the only minimizer of $b_2$ in $\conbod_0$) remains open, the characterization has been proven if one restricts to the class of $n$-polytopes. 

\begin{thm}[Simplex Characterization, \cite{SSZ1}]
\label{thm:polytope}
Let $P\subset \R^n$ be an $n$-polytope. Then $b_2(P)>1$, unless $P$ is an $n$-simplex.
\end{thm}

For reader's convenience, we sketch the ideas of the proof below. The proof uses the following fact (see \cite[Lemma $3.1$]{SSZ1}).
\begin{fact}
\label{cor:supportpolytope}
Let $P=\cap_{i=1}^N H^-(u_i,h_i)$ be an $n$-polytope, with outer normal vectors $u_1, ... , u_N$, and with support vector $(h_i)_{i\leq N}$. Let $i\leq N$, and let $|t|$ be small enough. Then the mixed surface area measures $\sigma_r:=S(P[n-1-r], P_{i,t}[r],.)$, all have the same support as $\sigma_0=S_P$.
\end{fact}
\begin{proof}[Proof of Theorem \ref{thm:polytope}].
Assume $P$ is an $n$-polytope and $b_2(P)=1$. In other words $F(L_1,L_2)\geq 0$ for all polytopes $L_1, L_2$, where
$F(A,B)=V_n(A,P[n-1])V_n(B,P[n-1])-V_n(A,B,P[n-2])V_n(P).$

Denote $u_1, ... , u_N$ the outer normal vectors of $P$, and $F_i=P^{u_i}$ the associated facets.
Thanks to the integral formula \eqref{eq_polytope} and to above Fact \ref{cor:supportpolytope}, the hypothesis $F(L_1,L_2)\geq 0$, when applied to $L_1=P_{i,t}$ and $L_2=P_{j,s}$, yields that, for any $|s|, |t|$ small, and $i, j \leq N$:
$$F(L_1,L_2)=\left(V(P)+\frac{t}{n}V_{n-1}(F_i)\right)\left( V(P)+\frac{s}{n} V_{n-1}(F_j)\right) - \left( V(P)+\frac{t}{n}V_{n-1}(F_i)+\frac{s}{n}V_{n-1}(F'_j)\right) V(P) \geq 0$$
where $V_{n-1}(F'_j):=V_{n-1}(P_{i,t}^{u_j},P^{u_j}[n-2])$ (this notation is justified by Minkowski's existence theorem \footnote{which tells us that there exists $P'=P'_{i,t}$ such that $V(L,P_{i,t},P[n-2])=V(L,P'[n-1])$ for all $L$}, though Minkowski's theorem is not needed for the current proof).

Once simplified, we obtain $s \left[ V(F_j) \left( 1+ \frac{t}{n} \frac{V(F_i)}{V(P)} \right) - V(F'_j)\right] \geq 0$, for any $|s|< \delta$ and $j\leq N$.
It follows that $V(F'_j)=\lambda_t V(F_j)$ for all $j\leq N$, with $\lambda_t=1+ \frac{t}{n} \frac{V(F_i)}{V(P)}=\frac{V(P_{i,t},P[n-1])}{V(P)}$.

Since $\sigma_1=S(P_{i,t},P[n-2],.)$ has same support as $S_P$ (cf. Fact \ref{cor:supportpolytope}), the $N$ equations $V(F'_j)=\lambda_t V(F_j)$ are enough to conclude that $\sigma_1=\lambda_t S_P=\lambda_t \sigma_0$, i.e. to conclude that $\sigma_1$ is proportional to $\sigma_0$.

One may rewrite $F(A,P_{i,t})\geq 0$ as $\int h_A d\sigma_1 \leq \lambda_t \int h_A d\sigma_0$, from which we have just deduced that $\sigma_1=\lambda_t \sigma_0$ (by taking $A=P_{j,s}$ for various $j\leq N$ and $s=\pm \delta$).
Aleksandrov-Fenchel inequality $V(A,P_{i,t},P_{i,t},P[n-3])V(A,P[n-1]) \leq V(A,P_{i,t},P[n-2])^2$ can be written as:
$$\left(\int h_A d\sigma_2\right) \left(\int h_A d\sigma_0\right) \leq \left( \int h_A d\sigma_1\right)^2.$$
Therefore, $\sigma_1=\lambda_t \sigma_0$ implies that $\int h_A d\sigma_2 \leq \lambda_t^2 \int h_A d\sigma_0$, for all $A$.

Similarly as above, by taking $A=P_{j,s}$ for $j\leq N$ and $s=\pm \delta$, and using Fact \ref{cor:supportpolytope}, one deduces from these inequalities that $\sigma_2=\lambda_t^2 \sigma_0$.

An induction argument (using some more Aleksandrov-Fenchel inequalities, and Fact \ref{cor:supportpolytope}) gives $\sigma_r=\lambda_t^r \sigma_0$, for all $r\leq n-1$.
In particular, one obtains $S_{P_{i,t}}=\sigma_{n-1}=\lambda_t^{n-1}\sigma_0=\lambda_t^{n-1} S_P$, which implies (by  Fact \ref{fact:surfaceunique}) that $P_{i,t}$ and $P$ are homothetic.

To conclude, it only remains to argue that a polytope with $n+2$ vertices (or more), always have a facet $F_i$, such that for all $t>0$ small enough, $P_{i,t}$ is not homothetic with $P$. This follows from the fact that if $P$ is not an $n$-simplex, then $P$ has a facet $F_i=P^{u_i}$, such that at least two vertices of $P$ lie outside of $F_i$. Denote $v_0, v_1$ two such vertices.

Note that when $t>0$ is small enough, all vertices of $P$ which lie outside of $F_i$, are still vertices of $P_{i,t}$. In particular, the distance $||v_1-v_0||$ is left unchanged by the perturbation. If $P_{i,t}$ was homothetic to $P$, since we have seen $\sigma_1=\lambda_t \sigma_0$, then any distance $||v_j-v_k||$ between two vertices of $P$, shall be multiplied by $\lambda_t>1$ under the perturbation: if we had $P_{i,t}=x+\lambda_t P$, then $v'_j=x+\lambda_t v_j$ would be the vertices of $P_{i,t}$, and in particular we should have $||v'_1-v'_0|| > ||v_1-v_0||$ (whereas we have $||v'_1-v'_0||=||v_1-v_0||$ since $v'_1=v_1$ and $v'_0=v_0$ here).
\end{proof}

Next, we wish to recall a few properties which are known to exclude a convex body $K$ from minimizing $b_2(.)$. The next proposition states that a $b_2$ minimizer must be indecomposable: this is \cite[Theorem 3.3]{SZ}.

\begin{prop}
\label{decomposable}
Assume $K$ is decomposable, i.e. $K=A+B$, with $A$ and $B$ not homothetic. Then $b_2(K)>1$.
\end{prop}


It turns out that Theorem~\ref{thm:polytope} and Proposition~\ref{decomposable} are merely special cases of a larger result. In \cite{SSZ2}, the concept of \textit{weak decomposability} was introduced.

\begin{dfn}
Let $K$ be a convex body with non-empty interior. We say that $K$ is weakly decomposable, if there exists $L$, non-homothetic to $K$, such that $supp(S_{K+L}) \subset supp(S_K)$.
\end{dfn}

\emph{Remark $1$} Any smooth convex body is weakly decomposable, because $supp(S_K)=\mathbb{S}^{n-1}$ (take any polytope for $L$). Likewise, any convex body such that $supp(S_K)$ contains an open subset of $\mathbb{S}^{n-1}$, is weakly decomposable.


\emph{Remark $2$} The simplex is the only $n$-polytope which is not weakly decomposable  (see Lemma~\ref{lem:polyisweaklynotsimplex} in the Appendix for a proof). This is why Theorem~\ref{thm:weakly} below implies Theorem~\ref{thm:polytope} above.

\emph{Remark $3$.} If $K$ is decomposable, say $K=A+B$ with $A, B$ not homothetic (and therefore, $A, K$ not homothetic), then $K$ is weakly decomposable. Indeed, choose $L=A$. Then:
$$
S_{K+L}=S_{2A+B}=\sum_{k=0}^{n-1} {{n-1}\choose k} 2^{n-1-k} \sigma_k \hspace{2mm} \text{ while }\hspace{2mm} S_K=\sum_{k=0}^{n-1} {{n-1}\choose k} \sigma_k
$$
where $\sigma_k=S(A[n-1-k], B[k], .)$ are the mixed surface area measures between $A, B$. It follows that $supp(S_K)=supp(S_{K+L})=\cup_{k=0}^{n-1} supp(\sigma_k)$.

The following theorem is due to Saroglou, Soprunov, and Zvavitch (see \cite[Theorem 5.7]{SSZ2}).

\begin{thm}
\label{thm:weakly}
Let $K$ be a weakly decomposable body. Then $b_2(K)>1$.
\end{thm}
In other words, like decomposability, weak decomposability is an excluding condition (when searching minimizers of $b_2(.)$). The proof reproduces the induction of Theorem~\ref{thm:polytope} replacing $P_{j,s}$ with Wulff shape perturbations. While the discrete nature of surface area measures of polytopes (and fact \ref{cor:supportpolytope}, i.e. nice behaviour of the support of surface area measures, with respect to Minkowski sum with a perturbation)  was a key ingredient for iteratively deducing $\sigma_r=\lambda_t^r \sigma_0$, in the above proof of Theorem \ref{thm:polytope}, in this other setting, fact \ref{cor:supportpolytope} 
wouldn't hold (if we also choose to replace $P_{i,t}$ with a Wulff-shape perturbation), but comes for free via the assumption $S_{K+L}<<S_K$, and the main ingredient to derive
$\sigma_r=\lambda_t^r \sigma_0$ (where now $\sigma_r=S(L[r], K[n-1-r],.)$), is then Alexandrov's variational lemma (see Lemmas \ref{lem:variationalmixed} and  \ref{lem:pointwiseCV}). For more details, see \cite{SSZ2}, p.16.

\noindent In view of the above theorem, it was asked in \cite{SSZ2} whether $n$-simplices are the only convex bodies $K\in \conbod_0$, which are not weakly decomposable. If true, this would solve Conjecture~\ref{conj:main}.

\begin{question} Is it true that the cone $C=Conv(e_n, B)$, where $B=\pi_{e_n^{\perp}}(B_2^n) \subset e_n^{\perp}$ is the unit ball in $e_n^{\perp}$, is  weakly indecomposable ?
\end{question}
 
\section{An excluding condition with isoperimetric ratios}
\label{sec:exclude}
If a property $\mathcal{P}$ is such that when $K$ has $\mathcal{P}$, then $K$ cannot be a minimizer of the $b_2(.)$ constant, we shall say that $\mathcal{P}$ is an excluding condition (i.e. for belonging to the set of minimizers). Theorem~\ref{thm:weakly} above, for example, shows that being weakly decomposable (which in particular includes being a polytope other than an $n$-simplex, or being decomposable) is an excluding condition. In this section we give a new excluding condition, which concerns bodies which at least one facet (but not necessarily polytopes). We will work in $\R^n$ with $n\geq 3,$ This condition implies that having infinitely many facets is an excluding condition  (see Theorem~\ref{thm:infinite} below). 

If $K$ is a $k$-dimensional convex body, with $k\geq 2$, then denote $\Isop(K)=\frac{1}{k}\frac{|\partial K|}{|K|}:=\frac{1}{k}\frac{\Vol_{k-1}(\partial K)}{\Vol_k(K)}$.

In \cite{S}, the author has shown the following.

\begin{prop}
\label{prop:isopcond}
Let $K$ be a convex body such that $K$ has a facet $F$ satisfying:
$\Isop(F)>\Isop(K)$. Then $b_2(K)>1$.
\end{prop}

Recall that for any (reversible) affine transform $T$, $b_2(K)=b_2(TK)$. On the other hand, the quantity
$$\sup_F \frac{\Isop(F)}{\Isop(K)}=\sup_{F} \frac{n \Vol_{n-2}(\partial F) \Vol_n(K)}{(n-1)\Vol_{n-1}(F) \Vol_{n-1}(\partial K)},$$ where the supremum is over the facets, is not affine-invariant. Thus, the above proposition immediately implies the following one.
\begin{prop}
\label{p_facets_image}
Let $K$ be a convex body such that one of its affine pairs $K'=TK$ has a facet $F'$ satisfying:
$$\frac{\Vol_{n-2}(\partial F')}{(n-1)\Vol_{n-1}(F')} > \frac{\Vol_{n-1}(\partial K')}{n\Vol_n(K')}.$$
Then $b_2(K)>1$.
\end{prop}
In other words, if $\max_T \sup_F \frac{\Isop(TF)}{\Isop(TK)} >1$,
where the max is over $T\in O(n)$ (and the supremum is over facets of $K$), then $b_2(K)>1.$ (see \cite{S} for some applications of the above proposition).

Moreover, Proposition~\ref{p_facets_image} yields a short proof of the following result, which states $\Omega_{n-1}$ (the set of outer-unit normals of facets of a convex body) being infinite (in the sense of cardinality) is an excluding condition.
\begin{thm}[Theorem 4.2 in \cite{SSZ2}]
\label{thm:infinite}
\label{infinitecond}
Let $K$ be a convex body with infinitely many facets. Then $b_2(K)>1$.
\end{thm}
\begin{proof}
If $K$ has infinitely many facets, then infinitely many of them will satisfy $\Isop(F)>\Isop(K),$
so that Theorem \ref{thm:infinite} follows from Proposition \ref{prop:isopcond}.

Let  $C$ be a $d$-dimensional convex body. Then 
$$\Isop(C)=\frac{1}{d}\frac{\Vol_{d-1}(\partial C)}{\Vol_d(C)}=\frac{1}{d} \frac{\Vol_{d-1}(\partial C)}{\Vol_d(C)^{d-1/d}}\Vol_d(C)^{-1/d} \geq \frac{1}{d}\frac{\Vol_{d-1}(\partial B_2^d)}{\Vol_{d}(B_2^d)^{d-1/d}}\Vol_d(C)^{-1/d}=\frac{\kappa_d^{1/d}}{\Vol_d(C)^{1/d}},$$ where we used the isoperimetric inequality, and where $\kappa_d$ denotes the volume of the $d$-dimensional euclidean ball.

It follows that $\Isop(F) \to +\infty$ when $\Vol_{n-1}(F) \to 0$. Since a convex body only has finite surface area measure, it follows that if $K$ has infinitely many facets, then all but finitely many of them will satisfy $\Isop(F)>\Isop(K)$.
\end{proof}

We conclude this section with two questions: the first one asks whether a discrete analog to Ball's reverse isoperimetric inequality holds, among polytopes. We denote by $\Delta$ an $n$-simplex. We denote $\mathcal{P}^n \subset \mathcal{K}_0^n$, the class of $n$-dimensional polytopes (living in $\R^n$).

\begin{question}
\label{quesa}
For $P\in \mathcal{P}^n$, is it true that $\max_{T\in O(n)} \max_F \frac{\Isop(TF)}{\Isop(TP)}$ is minimized when $P=\Delta$?
\end{question}
A related question is the following:

\begin{question}
\label{quesb}
Let $P\in \mathcal{P}^n$, other than a simplex. Do we have $\max_{T\in O(n)} \max_F \frac{\Isop(TF)}{\Isop(TP)} >1$ ?
\end{question}
In words, does there always exist an affine transform $T$ such that  $P'=TP$ has at least one facet $F'$ satisfying $\Isop(F')>\Isop(P')$ ?

\noindent If the latter question is in the affirmative, then Proposition \ref{p_facets_image} above yields an alternative proof of Theorem \ref{thm:polytope}. Note that a positive answer to Question \ref{quesb}, implies a positive answer to Question \ref{quesa} (though we don't know the exact value of $\max_{T\in O(n)} \max_F \frac{\Isop(TF)}{\Isop(T\Delta)}$, it follows from Proposition \ref{p_facets_image}, that $\max_{T\in O(n)} \max_F \frac{\Isop(TF)}{\Isop(T\Delta)} \leq 1$).

Another related question is the following: can one find a polytope $P$ such that, when $T\in O(n)$, the max-ratio $\max_F \frac{\Isop(TF)}{\Isop(TP)},$ is not minimized with $TP$ being in its John's position ? A good candidate is perhaps a polytope for which the transforms $T$ with $|det(T)|=1$ and such that $\Vol_{n-1}(\partial (TP))$ is minimized, are not the same as these which put $P$ in a John's position (i.e. these $T$ such that an ellipsoid contained in $TP$ which has maximal volume, is a Euclidean ball).

\section{An  excluding condition for minimizing $b_2$}
\label{sec:dual}

In \cite{SSZ1}, it is proved that a \emph{strict} point on the boundary (which occurs for instance if $\partial K$ has positive curvature somewhere) is an excluding condition. A strict point is a point $y\in \partial K$, such that no segment $L$ lying on the boundary, contains $y$. We state here an excluding condition which implies the one just mentionned.

 Recall the notation $supp(S_K)=\Omega=\Omega_0 \cup ... \cup \Omega_{n-1}$, with $\Omega_k=\{u\in \Omega: K^u \text{ is $k$-dimensional}\}$. Elements of $\Omega_0$ shall be called \emph{regular directions} of the boundary. If $y\in \partial K$, we remind the notation $\sigma(y)=\{u\in \Sph: y\in K^u\}$.
 
 Note that if $K$ has a strict point (on its boundary), then $\Omega_0\neq \emptyset$: this follows from the fact that $\sigma(z)\cap \Omega \neq \emptyset$, for any $z\in \partial K$. In other words the following condition is a slight strengthening of \cite[Theorem 4.1]{SSZ1}.
 
\begin{thm}
\label{positivecurv}
Assume $\Omega_0 \neq \emptyset$. Then $b_2(K)>1$.
\end{thm}

\noindent The proof uses standard compactness arguments, and the formula \eqref{eqn:integralformula}.

We first need to recall a folklore fact. Fix $K$ a convex body, and $y\in\partial K$. Then $\sigma(y)=\{u\in \Sph: y\in K^u\}=\{u\in \Sph: \langle y,u\rangle=h_K(u)\}$ is a non-empty, compact, geodesically convex subset of the sphere, and cannot contain both $u, -u$ (i.e. doesn't contain any pair of antipodal points). From these $4$ properties, one can deduce that $\sigma(y)$ is contained in an open half-sphere $\{v\in \mathbb{S}^{n-1} : \langle v,u_1\rangle >0\}$, whose center $u_1$ can be chosen to be in $\sigma(y)$.

\begin{lem}
\label{lem:folklore}
Let $K$ be a convex body and $y\in \partial K$. Then there exists $u_1\in \sigma(y)$, such that $\langle v,u_1\rangle >0$, for any $v\in \sigma(y)$.
\end{lem}


\begin{proof}[Proof of Theorem \ref{positivecurv}]
Let $u_0 \in \Omega_0$, and let $y\in K$ such that $K^{u_0}=\{y\}$. Then, from Lemma~\ref{lem:folklore} there exists $u_1\in \sigma(y)$ such that $\langle u_1, u\rangle >0$ for all $u\in \sigma(y)$.

For $\epsilon>0$, denote $K_{\epsilon}=K \cap \{x\in \R^n: \langle x, u_0 \rangle \leq  \langle y, u_0 \rangle -\epsilon \}$. Note that, if $\epsilon>0$ is small enough, then $\pi(K)=\pi(K_{\epsilon})$, where $\pi=\pi_{u_1^{\perp}}$ denotes the orthogonal projection onto $u_1^{\perp}$.

Indeed, if not, one could find a sequence $z_{\epsilon} \in K \setminus K_{\epsilon}$, such that $\pi(z_{\epsilon}) \notin \pi(K_{\epsilon})$, and moreover, $z_{\epsilon}$ is maximal in some direction $w_{\epsilon} \in \mathbb{S}^{n-1} \cap u_1^{\perp}$: $h_K(w_{\epsilon})=h_{\pi(K)}(w_{\epsilon})=\langle w_{\epsilon}, z_{\epsilon} \rangle=\langle w_{\epsilon}, \pi(z_{\epsilon})\rangle$. Passing to a subsequence, one can can assume $w_{\epsilon} \to w \in u_1^{\perp} \cap \mathbb{S}^{n-1}$, and $z_{\epsilon} \to z\in \partial K$. But, since $K^{u_0}=\{y\}$, and $\langle z_{\epsilon} ,u_0 \rangle > h_K(u_0)-\epsilon$, we have $z=y$, ie $\lim z_{\epsilon}=y$.
By continuity, one deduces $y\in K^w$, in other words $w\in \sigma(y)$, which contradicts the fact that $\langle u_1,v \rangle >0$ for all $v\in \sigma(y)$.

Next, fix $\epsilon>0$ sufficiently small, so that $\pi(K)=\pi(K_{\epsilon})$. Then (using \eqref{eqn:proj1}), $V([0,u_1], K[n-1])=V(K_{\epsilon}, [0,u_1], K[n-2])=\frac{1}{n}\Vol_{n-1}(\pi(K))$, so that to establish 
$$
V(K_{\epsilon}, [0,u_1], K[n-2]) V(K) > V(K_{\epsilon},K[n-1]) V([0,u_1], K[n-1]),
$$
it is sufficient to prove that $V(K)>V(K_{\epsilon}, K[n-1])$. This inequality follows from the integral formula \eqref{eqn:integralformula}, and from the assumption $u_0 \in supp(S_K)$. Indeed,  $h_{K_{\epsilon}}(u_0)=h_K(u_0)-\epsilon$ (by definition of $K_{\epsilon}$), and so by continuity of $h_{K_{\epsilon}}-h_K$, there exists some $\delta>0$ such that $(h_K-h_{K_{\epsilon}})(u)>\frac{\epsilon}{2}$ if $u\in U(u_0, \delta)=\{u\in \mathbb{S}^{n-1}: \text{dist}(u,u_0)\leq \delta\}$. Moreover $(h_K-h_{K_{\epsilon}})(u)\geq 0$ for all $u\in \mathbb{S}^{n-1}$. Hence, by non-negativity of the measure $S_K$, one finds:
$$
V(K)-V(K_{\epsilon},K[n-1])=\frac{1}{n} \int (h_K-h_{K_{\epsilon}}) dS_K  \geq \frac{\epsilon}{2n} S_K\left(U(u_0, \delta)\right) >0,
$$
where positivity is due to the assumption $u_0 \in supp(S_K)$.
%
%
%
%
\end{proof}

The next statement is the main result of this paper. It is an excluding condition, which deals with $\Omega_{n-2}$ rather than with $\Omega_0$ or $\Omega_{n-1}$.

\begin{reptheorem} {thm:main}
Assume $S_K(\Omega_{n-2})>0$. Then $b_2(K)>1$.
\end{reptheorem}

 As an example: in $\R^3$, if $B=\pi_{e_3^{\perp}}(B_2^3)$ is the unit ball in $e_3^{\perp}$, and if $C=Conv(B, e_3)$ denotes the convex hull of $B$ and $e_3$, then $\Omega_1$ is a circle (lying on $\mathbb{S}^2$). Hence $b_2(C)>1$.

Before we jump to the proof of Theorem \ref{thm:main}, let us show how to use it to prove Corollary~\ref{cor:intro}.
\begin{proof}[Proof of Corollary~\ref{cor:intro}]
Consider $K\in\mathcal{K}^3$. Let $\Omega=\Omega_0 \cup \Omega_1 \cup \Omega_2$ be the support of the surface area measure $S_K$.
If $\Omega_0 \neq \emptyset$, then $b_2(K)>1,$ due to Theorem~\ref{positivecurv}. Thus, we may assume $\Omega_0=\emptyset$. If $\Omega_2$ is infinite, then $b_2(K)>1,$ due to Theorem~\ref{infinitecond}. Thus, we may assume $\Omega_2$ is a finite set. If $S_K(\Omega_1)>0$, then Theorem~\ref{thm:main} yields $b_2(K)>1$. Hence, we may assume $S_K(\Omega_1)=0$. Consequently, $S_K(\Omega)=S_K(\Omega_2)$ and $\Omega_2$ is a finite set, thus a closed set, so (by definition of the support of a measure), one gets $\Omega=\Omega_2$, which means that $K$ is a polytope.
From Theorem~\ref{thm:polytope}, we know the only $n$-polytopes satisfying $b_2(P)=1$ are simplices.
\end{proof}

Note that a straightforward adaptation of the above proof would show that $b_2(K)=1$ characterizes the $4$-simplex among $4$-dimensional convex bodies, if one could prove that $S_K(\Omega_{n-3})>0$ enforces $b_2(K)>1$. And similarly, if $S_K(\Omega_{n-k})=0$ was a necessary condition (for having $b_2(K)=1$), for $k=2,3,4$, then the above proof would yield that $\Delta_5$ is the only minimizer of $b_2$ in $\mathcal{K}_0^5$.

\vspace{2mm}
Let us mention that the above corollary already appeared in \cite{SSZ2}. Indeed, when proving that $b(K)=1$ characterizes the simplex, Saroglou, Soprunov and Zvavitch obtained a stronger characterization, namely that $b'(K)=1$ characterizes the simplex, where:
$$
b'(K)=\max_{L_1, ... , L_{n-1}} \frac{V(L_1, ... , L_{n-1},K) V(K)}{V(K,K, L_2, ... , L_{n-1})V(L_1, K[n-1])}.
$$
Since $b'(K)\leq b(K)$ for all $K$, and the inequality is sometimes strict (for the cube for instance), the condition $b'(K)=1$ is weaker (and so the characterization is stronger). Note that when $n=3$, then $b'(K)=b_2(K)$ (for all $K$), which is why Corollary~\ref{cor:intro} is a consequence of \cite[Theorem 1.1]{SSZ2}.




We now jump to the proof of Theorem~\ref{thm:main}.

\begin{proof}[Proof of Theorem~\ref{thm:main}]
Let $n\geq 3$. Assume $S_K(\Omega_{n-2})>0$. If $\Omega_0\neq \emptyset$, then by Theorem~\ref{positivecurv}, we know $b_2(K)>1$: hence we can assume $\Omega_0=\emptyset$. Similarly, we can assume $\Omega_{n-1}$ is a finite set.
$$\text{Define  }  \Omega_{\delta}:=  \{u\in \Omega_{n-2}: \Vol_{n-2}(K^u) \geq \delta \} , \hspace{2mm} \text{so that }\hspace{3mm} \Omega_{n-2}=\cup_{\delta>0} \Omega_{\delta}.
$$
Fix $\delta>0$ such that $S_K(\Omega_{\delta})>0$, and let $V_0=\Omega_{\delta}$. Note that $V_0=\cup_{\epsilon} \left( V_0 \setminus \Omega_{n-1}^{\epsilon}\right)$, where $A^{\epsilon}=\{u\in \mathbb{S}^{n-1}: dist(u,A)\leq \epsilon\}$, because $\Omega_{n-1}$ is finite.
Fix $\epsilon >0$ such that $S_K(V_0 \setminus \Omega_{n-1}^{\epsilon})>0$. Set $V:=V_0 \setminus \Omega_{n-1}^{\epsilon}$.

Let $\epsilon_0=\frac{1}{10} \epsilon $, and cover $V \subset \mathbb{S}^{n-1}$ with caps of radius $\epsilon_0$, centered at points in $V$. Among these caps, at least one, call it $U_0=U(z_0, \epsilon_0)$, is such that $S_K(V \cap U_0)>0$. Once $U_i$ is defined, cover $V$ with caps of radius $\epsilon_{i+1}=\frac{1}{4} \epsilon_i$. Among them, one may choose $U_{i+1}$, such that $U_{i+1}\cap U_i \neq \emptyset$, and moreover $S_K(V\cap U_{i+1})>0$. 

We claim that $S_K(2U_k)\to 0$, where we remind that $2U_k=U(z_k, 2\epsilon_k)$, if $U_k=U(z_k, \epsilon_k)$. 

Indeed, set $U=\bigcap_m \bigcup_{k\geq m} (2U_k)$. Since $\epsilon_{i+1}=\frac{1}{4}\epsilon_i$ and $U_{i+1}\cap U_i \neq \emptyset$, note that $2U_{i+1} \subset 2U_i$: in particular $U$ is a singleton, and if $U=\{u_0\}$, we have $u_0\in 2U_0$. By construction, $2U_0 \cap \Omega_{n-1}=\emptyset$, therefore $u_0\notin \Omega_{n-1}$, in other words $K^{u_0}$ is at most $(n-2)$-dimensional, and $S_K(U)=S_K(\{u_0\})=0$. Since $S_K(U)\geq \limsup_k S_K(2U_k)$, it readily follows that $S_K(2U_k)\to 0$.

%

Define $f_i: \mathbb{S}^{n-1} \to [0,1]$, continuous, such that $f_i=1$ on $U_i$, and $f_i=0$ outside $2U_i$ \footnote{if $U_k=U(z_k, \epsilon_k)$ is cap centered at $z_k$, then $2U_k:=U(z_k ,2\epsilon_k)$}. Define Wulff-shape perturbations of $K$ as follows: $L_i(t)=W(h_K+t f_i)$. According to Lemma~\ref{lem:pointwiseCV}, one has 
$$
\frac{h_{L_i(t)}(u)-h_K(u)}{t} \to 1 \hspace{3mm} \text{ for $S_K$-almost every $u\in U_i$}.
$$
Call $W_i:= \{u\in U_i: \frac{h_{L_i(t)}(u)-h_K(u)}{t} \to 1 , \text{ as $t\to 0$}\}$. Then almost sure convergence tells us that $S_K(W_i\cap V)=S_K(U_i \cap V)>0$, so that $W_i\cap V \neq \emptyset$. Fix $u_i \in W_i \cap V$. Because $u_i \in V$, we know that $K^{u_i}$ has dimension $(n-2)$ (and moreover that $\Vol_{n-2}(K^{u_i})\geq \delta$). Let $v_i\in \mathbb{S}^{n-1} \cap u_i^{\perp}$ be such that $K^{u_i} \subset v_i^{\perp}$ (it is uniquely determined, up to $\pm1$ sign). Set $M_i$ to be the square: $M_i=[0,u_i]+[0,v_i]$, so that $M_i^{u_i}$ is a translate of $[0,v_i]$.

The next claim finishes the proof.

\begin{claim}
Let $F(A,B)=V(A,B,K[n-2])V(K)-V(A,K[n-1])V(B,K[n-1])$, where $A,B \in \conbod$. Then, if $k>1$ is large enough, and if $t\in (0,t_k)$ is small enough, one has $F(L_k(t), M_k)>0$.
 \end{claim}

\noindent To prove this claim, let us rewrite:
$$
F(A,B)=\left[ V(A,B,K[n-2])-V(K,B,K[n-2])\right] V(K) - \left[ V(A,K[n-1])-V(K)\right] V(B,K[n-1]) ,
$$
so that $F(L_k(t), M_k)=A_k(t) V(K) -B_k(t) V(M_k, K[n-1])$, where
$$
A_k(t)=V(L_k(t), M_k, K[n-2])-V(K, M_k, K[n-2]) \text{ and } B_k(t)=V(L_k(t), K[n-1])-V(K).
$$
Now, $B_k(t)=\frac{1}{n} \int (h_{L_k(t)}-h_K) dS_K \leq \frac{t}{n} \int f_k dS_K \leq \frac{t}{n} S_K(2U_k)$, since $supp(f_k)\subset 2U_k$, and $\max f=1$.
%

Meanwhile, $V(M_k, K[n-1])=V([0,u_k], K[n-1])+V([0,v_k], K[n-1])$, and so, setting $C_K=\max_{u\in \mathbb{S}^{n-1}} \Vol_{n-1}\left(\pi_{u^{\perp}}(K) \right)$, one gets
 $V(M_k, K[n-1]) \leq \frac{2}{n} \max_{u\in \mathbb{S}^{n-1}} \Vol_{n-1}\left(\pi_{u^{\perp}}(K) \right)= \frac{2}{n} C_K$.
 
On the other hand, let us denote $\sigma=S(M_k, K[n-2], .)$ the mixed surface area measure between $M_k$ and $K$, so that:
$$
A_k(t)=V(L_k(t), M_k, K[n-2])-V(K, M_k, K[n-2]) =\frac{1}{n} \int (h_{L_k(t)}-h_K)(u) d\sigma(u) .
$$
Then $A_k(t) \geq \frac{1}{n} (h_{L_k(t)}-h_K)(u_k) S(M_k, K[n-2], u_k)$, but note that
$$
S(M_k, K[n-2],u_k)=V_{n-1}(M_k^{u_k}, K^{u_k}[n-2])=V_{n-1}([0,v_k], K^{u_k}[n-2])=\frac{1}{n-1} \Vol_{n-2}(\pi_{v_k^{\perp}}(K^{u_k})).
$$
$$ \text{Hence, } S(M_k, K[n-2],u_k) =\frac{1}{n-1} \Vol_{n-2}(\pi_{v_k^{\perp}}(K^{u_k}))=\frac{1}{n-1} \Vol_{n-2}(K^{u_k}) \geq \frac{\delta}{n-1}$$
where the second equality is by choice of $v_k$, and where the inequality results from $u_k \in V \subset \Omega_{\delta}$.

Since $u_k$ was also chosen to be in $W_k$, note that $(h_{L_k(t)}-h_K)(u_k) \sim t$ as $t\to 0$. Therefore, if $t$ is small enough (say $t\in (0,t_k)$), one has:
$A_k(t) \geq \frac{\delta}{2n (n-1)} t$.

Putting the inequalities together yields that for any $k$, and for any $t\in (0,t_k)$:
$$F(L_k(t),M_k) \geq \left( \frac{\delta}{2n (n-1)} V(K) - \frac{2}{n^2} C_K S_K(2U_k) \right) t .$$
Since $S_K(2U_k)\to 0$, the claim follows.
\end{proof}

\subsection{Characterization of $\Delta$ as the only minimizer of $b(K)$}
\label{sec:thm9}
This section is devoted to the following theorem.
\begin{reptheorem}{thm:b(K)}{\cite{SSZ2}}
Let $K$ be a convex body such $b(K)=1$. Then, $K$ is an $n$-simplex.
\end{reptheorem}

Although this theorem is a corollary of \cite[Theorem 1.1]{SSZ2}, we provide here a more simple proof, following the same notations as in previous sections.  The proofs of Theorems \ref{thm:b(K)} and \ref{thm:main} are indeed very similar.
\begin{proof}
Assume $K$ is not an $n$-simplex. Thanks to Theorem~\ref{infinitecond}, we may assume $\Omega_{n-1}$ is a finite set. If $K$ is a polytope, we already know that $b(K)\geq b_2(K)>1$. Thus we may assume $\Omega \neq \Omega_{n-1}$.  If there are regular points on the boundary of $K$, namely if $\Omega_0\neq \emptyset$, we know that $b(K)\geq b_2(K) >1$: we may assume $\Omega_0=\emptyset$.


If $\Omega_{n-1}$ is empty (no facets), we set $V:=\Omega$. Else, note that $\Omega \setminus \Omega_{n-1}=\cup_{\epsilon>0} \left(\Omega \setminus \Omega_{n-1}^{\epsilon} \right)$, using that $\Omega_{n-1}$ is finite. Therefore $S_K(\Omega)=\lim_{\epsilon} S_K\left(\Omega \setminus \Omega_{n-1}^{\epsilon} \right)$, and thus $S_K\left(\Omega \setminus \Omega_{n-1}^{\epsilon} \right) >0$ when $\epsilon$ is small enough. We fix such an $\epsilon$, and set $V:=\Omega \setminus \Omega_{n-1}^{\epsilon}$.

One may cover $V\subset \mathbb{S}^{n-1}$ with caps of radius $\epsilon_0=\frac{1}{3}{\epsilon}$, centered at points in $V$: among these caps, one may choose one, call it $U_0=U(z_0,\epsilon_0)$ (with $z_0\in V$), such that $S_K(V\cap U_0)>0$. One may cover $U_0$ with caps of radius $\epsilon_1=\frac{1}{4} \epsilon_0$: among these, choose $U_1$ such that $S_K(U_1 \cap V)>0$. And so on. This defines a sequence of caps $(U_k)$, which satisfies $U_k \cap U_{k+1}\neq \emptyset$, and $2U_{k+1} \subset 2U_k \subset 2U_0$. Let $u_0$ be the limitpoint of the sequence: $\{u_0\}=\bigcap_m \bigcup_{k\geq m} (2U_k)$. Then $u_0\in 2U_0=U(z_0, 2\epsilon_0)$, and so $\text{dist}(u_0,\Omega_{n-1})\geq \text{dist}(z_0,\Omega_{n-1})-2\epsilon_0 \geq \epsilon_0>0$, so that $u_0 \notin \Omega_{n-1}$ and hence $S_K(\{u_0\})=0$.

It follows that $S_K(2U_k)\to 0$.


Then define a sequence of continuous functions as follows: let $f_0: \mathbb{S}^{n-1} \to [0,1]$ be supported on $2U_0$ (recall that if $U_0=U(u_0, \delta)$, then $2U_0:=U(x,2\delta)$), such that $f_0=1$ on $U_0$. Similarly let $f_i: \mathbb{S}^{n-1} \to [0,1]$ be supported on $2U_i$, such that $f_i=1$ on $U_i$. By construction $2U_i \cap \Omega_{n-1}=\emptyset$ for all $i$. Then define the Wulff-shape perturbations of $K$: $L_k(t)=W(h_K+tf_k)$.

Thanks to Lemma~\ref{lem:pointwiseCV}, we know that $\frac{h_{L_k(t)}(u)-h_K(u)}{t} \to 1$ for $S_K$-almost every $u\in U_k$. Call $W_k\subset U_k$, the set of vectors $u$ such that this convergence holds. Then $S_K(V \cap W_k)=S_K(V\cap U_k) >0$, proving that $V \cap W_k$ is non-empty, for each $k$. Fix $u_k \in V\cap W_k$.

Define $M_k=B_2^n \cap H_k^-$ where $H_k^-=\{x\in \R^n: \langle x, u_k \rangle \leq 0\}$. Hence $M_k^{u_k} \approx B_2^{n-1}$ is the unit ball in $u_k^{\perp}$. The next claim finishes the proof.
\begin{claim}
If $k>1$ is large enough, and if $0<t<t_k$ is small enough, then:
$$V(L_k(t), M_k[n-1])V(K) > V(K, M_k[n-1]) V(L_k(t), K[n-1]).$$
\end{claim}
Let $F(A,B)=V(A,B[n-1])V(K)-V(A, K[n-1])V(K, B[n-1])$, so that we aim at showing that $F(L_k(t), M_k)>0$ when $k$ is large enough, and for small enough $t$.

Rewrite $F(L_k(t), M_k)=A_k(t) V(K) - B_k(t) V(K, M_k[n-1])$, where
$A_k(t)=V(L_k(t),M_k[n-1])-V(K,M_k[n-1]))$ and $B_k(t)=V(L_k(t), K[n-1])-V(K)$.
Recall that $L_k(t)$ being a Wulff-shape, we have $h_{L_k(t)}\leq h_K +t f$ (for all $u\in \mathbb{S}^{n-1}$). Therefore
$$
B_k(t)\leq \frac{t}{n} \int f_k(u) dS_K(u) \leq \frac{t}{n} S_K(2U_k) \text{ since $f_k\leq 1$ and $supp(f_k)\subset 2U_k$.}
$$

Denote $w(K)$ the mean-width of $K$ : one may define the mean-width by $w(K)=\int_{\Sph} h_K(u) du$, where $h_K$ denotes the support function of $K$, and where $du$ denotes the Haar measure on the sphere. Then, $V(K, M_k[n-1]) \leq V(K, B_2^n[n-1])=\kappa_n w(K)$ is upper-bounded by a constant (since $K$ is fixed). Now denote $S_{M_k}$, the surface area measure of the half-ball $M_k$. Recall that $S_{M_k}(\{u_k\})=\kappa_{n-1}$, since $M_k^{u_k} \approx B_2^{n-1}$. For the sake of brievity, let us write $L$ for $L_k(t)$ in what follows.
\begin{equation}
\label{eqn:final}
A_k(t)=\frac{1}{n} \int (h_L-h_K)(u) dS_{M_k}(u) = \frac{1}{n} (h_L-h_K)(u_k) S_{M_k}(u_k)=\frac{\kappa_{n-1}}{n} (h_{L}(u_k)-h_K(u_k)).
\end{equation}
The second equality comes from the choice of $f_k$, which yields that, for any $t>0$, one has  $h_{L_k(t)}(v)=h_K(v)$ for all $v$ outside $2U_k$, in particular for all $v \in \mathbb{S}^{n-1} \cap H_k^-$, and from the fact that, $M_k$ being a half-ball, $supp(S_{M_k})=\{u_k\} \cup \left(   \mathbb{S}^{n-1} \cap H_k^-\right)$.

\noindent Set $c_n=\frac{\kappa_{n-1}}{2n}$. Then
since $u_k \in W_k$, the equality \eqref{eqn:final} yields $A_k(t)\sim 2 c_n t$. In particular, for any small enough $t$, say $t\in (0,t_k)$, we have $A_k(t) \geq c_n t$. It results from this lower bound and from the previous two upper bounds, and from $S_K(2U_k) \to 0$, that for any large enough $k$, and any $t\in (0,t_k)$:
\begin{equation}
  F(L_k(t), M_k)=A_k(t) V(K)-B_k(t)V(K,M_k[n-1]) \geq \left( c_n V(K)- \frac{\kappa_n w(K)}{n} S_K(2U_k) \right) t >0.
\end{equation}
\end{proof}



\section{Appendix and remarks}
\label{sec:appendix}
In this section we give a short proof of Lemma~\ref{lem:rectangle}, which was needed to establish that $b_2(C)=\frac{n}{n-1}$, where $C=[0,1]^n$. We wish to thank Eli Putterman for  this nice proof. Then, we explain why the simplex is the only $n$-polytope which is not weakly decomposable, which implies that Theorem~\ref{thm:polytope} can be seen as a consequence of Theorem~\ref{thm:weakly}.

When $L$ is a segment, we denote by $|L|$ its length.
\begin{lem}
\label{lem:rectangle}
Let $A, B$ be two convex bodies in $\R^2$. Let $\pi_i(D)$ be the orthogonal projection of $D$ onto the axis $\R e_i$ ($i=1,2$). Then $V_2(A,B) \leq \frac{1}{2} \left( |\pi_1(A)| | \pi_2(B)|+|\pi_2(A)| |\pi_1(B)| \right)$.
\end{lem}
\begin{proof}
Let $L_1(A):= \pi_1(A)$, and $L_2(A):=\pi_2(A)$ be the projections of $A$ onto $\R e_1$, and onto $\R e_2$, respectively.
Similarly, denote $L_i(B)$ the projection of $B$ onto $\R e_i$, $i=1,2$.

Note that $A\subset R_A:=L_1(A)+L_2(A)$, and similarly $B\subset R_B:=L_1(B)+L_2(B)$. By monotonicity of mixed volume, one has $V_2(A,B) \leq V_2(R_A, R_B)$.

Since $R_A$ and $R_B$ are rectangles, the mixed volume  $V_2(R_A, R_B)$ is easily computed:
\begin{align*} 
V_2(R_A, R_B)&=V_2(L_1(A)+L_2(A) , L_1(B)+L_2(B))=V_2(L_1(A), L_2(B))+V_2(L_2(A), L_1(B))
\\
&=\frac{1}{2} \left( |\pi_1(A)| | \pi_2(B)|+|\pi_2(A)| |\pi_1(B)| \right).
\end{align*}
\end{proof}

\begin{lem}
\label{lem:polyisweaklynotsimplex}
Let $P \subset \R^n$ be a full-dimensional polytope. Then $P$ is weakly decomposable, unless $P$ is an $n$-simplex.
\end{lem}

\begin{proof}
Let $P$ be a polytope with $n+2$ vertices or more. Then there exists a facet $F$, such that at least two vertices, call them $x$ and $y$, lie outside $F$. Without loss of generality, we can assume $x$ and $y$ are neighbors, i.e. that $[x,y]$ is an edge of $P$. Let $u_1$ be the outer normal vector to $F$. Assume $P$ has $N$ facets in total and $P=\cap_{j=1}^N H^-(u_j, h_j),$ where we recall the notation: $H^-(u,b)=\{x: \langle x,u \rangle \leq b\}$.

Recall that for any $t>0$, polytopes $P_{1,t}$ were defined by $P_{1,t}=H^-(u_1, h_1+t) \cap \left( \bigcap_{j>1} H^-(u_j, h_j) \right)$, and that for small $t$, $P_{1,t}$ has same outer normal vectors as $P$, and additionally its support vector is given by: $h_{P_{1,t}}(u_j)=h_j+t\delta_{1,j}$, $j\leq N$.

\noindent Therefore $V(P[n-1],P_{1,t})=V(P)+\frac{t}{n} \Vol_{n-1}(P^{u_1})=V(P)+\frac{t}{n} \Vol_{n-1}(F)$. If (for contradiction) $P_{1,t}$ was homothetic to $P$, then there would be $x\in \R^n$, such that $P_{1,t}=x+\lambda P$, with (necessarily) $\lambda=V(P[n-1],P_{1,t})/V(P)=1+ \frac{t}{n} \frac{\Vol_{n-1}(F)}{\Vol_{n}(P)} >1 $ ( where $\lambda>1$ if $t>0$).

While a vertex $v\in F$ is moved to a nearby vertex $v_t \in F_t =P_{1,t}^{u_1}$, it follows from the definition of $P_{1,t}$ that the vertices $x$ and $y$ don't move: they are also vertices of $P_{1,t}$. Moreover, the edge $[x,y]$ can be separated from $F$ by some affine hyperplane: denote $H^+$ the half-space delimited by this hyperplane, where $[x,y]$ belongs. Then $P \cap H^+=P_{1,t}\cap H^+$ for any $t>0$. In particular, $[x,y]$ is also an edge of $P_{i,t}$.

But if $P_{i,t}$ was homothetic to $P$, then to every edge $e$ of $P$, would correspond an edge $e_t$ of $P_{1,t}$, whose length would be $|e_t|=\lambda |e|> |e|$. However, $[x_t, y_t]=[x,y]$ has same length as $[x,y]$.

\vspace{3mm}
Conversely, let's prove that $\Delta$ is not weakly decomposable (where $\Delta$ is an $n$-simplex). Assume $K\in \mathcal{K}_0^n$ is such that $S_{K+\Delta} << S_{\Delta}$. Since $S_K<< S_{K+\Delta}$ always holds, we deduce that $supp(S_K)\subset supp( S_{\Delta})=:\{u_1, ... , u_{n+1}\}$; in particular, $supp(S_K)$ is finite and $K$ is a polytope. Since $K$ is bounded, $supp(S_K)$ has at least $n+1$ elements: therefore $supp(S_K)=supp(\Delta)$, which implies that $K$ is a simplex, with same outer normal unit vectors as $\Delta$, hence $K$ and $\Delta$ are homothetic. 
\end{proof}
%

\subsection{A remark about the cube and the $l_2$-ball}
\label{appendix:cube}

In Claim \ref{claim:cube}, we have shown that $b_2(C)=\frac{n}{n-1}$, where $C=[0,1]^n$ is the unit cube. Since $b(C)=n=\max_K b(K)$, one may wonder whether $C$ is also maximal for $b_2$, but the $l_1$-ball gives a counterexample (see Lemma \ref{fenchel} and the example afterwards). One may then ask whether $b_2(C)$ is maximal for $K\in \mathcal{Z}_0^n$, where $\mathcal{Z}_0^n$ denotes the class of $n$-dimensional zonoids (a zonotope $Z\subset \R^n$ is a sum of segments, and a zonoid is a limit of zonotopes, for Hausdorff convergence\footnote{ in other words: $\mathcal{Z}^n$ is the smallest (topologically) closed subclass of $\mathcal{K}^n$ containing all segments, and closed under Minkowski addition}, if $\mathcal{Z}^n \subset \mathcal{K}^n$ denotes the class of zonoids in $\R^n$, then denote $\mathcal{Z}_0^n=\mathcal{Z}^n \cap \mathcal{K}_0^n$). 

Some authors (see \cite{FMMZ}) have recently been investigating affine-invariant constants related to $b_2$. Namely, define $b_{\mathcal{Z}}(K)=\sup_{(A,B)} \frac{V(A,B,K[n-2])V(K)}{V(A,K[n-1])V(B,K[n-1])}$, where the supremum is over pairs of zonoids (or equivalently, over pairs $(A,B)\in (\mathcal{Z}^n \setminus \R^n)^2)$.
Note that $b_{\mathcal{Z}}(K) \leq b_2(K)$.

The calculation of claim \ref{claim:cube} shows that $b_2(C)=b_{\mathcal{Z}}(C)$, but it is unclear whether this is necessarily the case for any $K$, even restricting to $K\in \mathcal{Z}_0^n$. As an example, consider $K=B_2^n$, the Euclidean ball. One may easily check (applying \eqref{eqn:proj1}, \eqref{eqn:proj2}, and multilinearity) that
$b_{\mathcal{Z}}(B_2^n)=\frac{n}{n-1} \frac{\kappa_n \kappa_{n-2}}{\kappa_{n-1}^2}=\frac{W_{n-2}}{W_{n-1}}$ (so in particular, $b_{\mathcal{Z}}(B_2^n)<\frac{n}{n-1}=b_{\mathcal{Z}}(C$)), where $W_n=\int_0^{\pi/2} \cos(\theta)^n d\theta$.

We note that \cite[Theorem 1.2]{AFO} suggests that $b_2(B_2^n)=b_{\mathcal{Z}}(B_2^n)$, but whether or not equality holds seems unknown. More generally, whether or not $b_2(K)=b_{\mathcal{Z}}(K)$ for all $K\in \mathcal{Z}_0^n$, (or even for all $K\in \mathcal{K}_0^n$), is unknown.

According to \cite[Theorem 3.9]{FMMZ},  $\max_{K\in \mathcal{Z}_0^n} b_{\mathcal{Z}}(K)=\frac{n}{n-1}$ is equivalent to 
$$\frac{\Vol_{n-2}(\partial(\pi_{u^{\perp}} K))}{\Vol_{n-1}(\pi_{u^{\perp}}K)} \leq \frac{\Vol_{n-1}(\partial K)}{\Vol_n(K)} , \hspace{3mm} \text{for all } K \in \mathcal{Z}_0^n \text{ and all } u\in \Sph.$$
The above inequality is false for some $K\in \mathcal{K}_0^n$ (see the cases of equality discussed in \cite{FGM} after [Proposition 3.1] therein), but may hold within the restricted class of zonoids.



\end{document}